\documentclass[final]{siamltex}

\usepackage{amssymb,graphics,graphicx}
\usepackage{upquote}
\def\Im{{\rm Im}\kern 2pt}

\title{Trigonometric interpolation and quadrature
in perturbed points}

\author{Anthony P. Austin\thanks{\texttt{austina@anl.gov},
Mathematics and Computer Science Division, Argonne National
Laboratory, 9700 S. Cass Ave., Lemont, IL 60439}\and
Lloyd N.~Trefethen\thanks{\texttt{trefethen@maths.ox.ac.uk},
Mathematical Institute, University of Oxford, Oxford, OX2 6GG,
UK.} \newline\indent Both authors were supported by the European
Research Council under the European Union's Seventh Framework
Programme (FP7/2007--2013)/ERC grant agreement no.\ 291068.\ \ The
views expressed in this article are not those of the ERC or the
European Commission, and the European Union is not liable for any
use that may be made of the information contained here.  The work
of the first author was additionally supported by the U.S.\
Department of Energy, Office of Science, under contract number
DE-AC02-06CH11357. \newline\indent The submitted manuscript has
been created by UChicago Argonne, LLC, Operator of Argonne National
Laboratory (``Argonne''). Argonne, a U.S.\ Department of Energy Office
of Science laboratory, is operated under Contract No.\
DE-AC02-06CH11357.  The U.S. Government retains for itself, and
others acting on its behalf, a paid-up nonexclusive, irrevocable
worldwide license in said article to reproduce, prepare derivative
works, distribute copies to the public, and perform publicly and
display publicly, by or on behalf of the Government.  The
Department of Energy will provide public access to these results
of federally sponsored research in accordance with the DOE Public
Access Plan. \texttt{http://energy.gov/downloads/doe-public-access-plan}.}

\begin{document}
\maketitle

\begin{abstract}
The trigonometric interpolants to a periodic function $f$
in equispaced points converge if $f$ is Dini-continuous,
and the associated quadrature formula, the trapezoidal rule,
converges if $f$ is continuous.  What if the points are perturbed?
With equispaced grid spacing $h$, let each point be perturbed by
an arbitrary amount $\le \alpha h$, where $\alpha\in [\kern .5pt
0,1/2)$ is a fixed constant.  The Kadec 1/4 theorem of sampling
theory suggests there may be be trouble for $\alpha\ge 1/4$.
We show that convergence of both the interpolants and the
quadrature estimates is guaranteed for all $\alpha<1/2$ if $f$
is twice continuously differentiable, with the convergence rate
depending on the smoothness of $f$.  More precisely it is enough
for $f$ to have $4\alpha$ derivatives in a certain sense, and we
conjecture that $2\alpha$ derivatives is enough.  Connections with
the Fej\'er--Kalm\'ar theorem are discussed.
\end{abstract}


\pagestyle{myheadings}
\thispagestyle{plain}
\markboth{AUSTIN AND TREFETHEN}
{INTERPOLATION AND QUADRATURE IN PERTURBED POINTS}

\section{Introduction and summary of results}

The basic question of robustness of mathematical algorithms is,
what happens if the data are perturbed?  Yet little
literature exists on the effect on interpolants, or on quadratures,
of perturbing the interpolation points.

The questions addressed in this paper arise in two almost equivalent
settings: interpolation by algebraic polynomials (e.g., in
Gauss or Chebyshev points) and periodic interpolation by
trigonometric polynomials (e.g., in equispaced points).
Although we believe essentially the same results hold in
the two settings, this paper deals with just the trigonometric case.
Let $f$ be a real or complex function on $[-\pi,\pi)$, which we take
to be $2\pi$-periodic in the sense that any assumptions of
continuity or smoothness made for $f$ apply periodically at $x=-\pi$ as
well as at interior points.
For each $N\ge 0$, set $K = 2N+1$, and consider the
centered grid of $K$ equispaced points in $[-\pi,\pi)$,
\begin{equation}
x_k^{} = kh, \quad -N \le k \le N, \quad h = {2\pi\over K}.
\end{equation}
There is a unique degree-$N$ trigonometric
interpolant through the data $\{f(x_k^{})\}$, by which we mean a function
\begin{equation}
t_N^{}(x) = \sum_{k=-N}^N c_k^{} e^{ikx}
\end{equation}
with $t_N^{}(x_k^{}) = f(x_k^{})$ for each $k$.
If $I$ denotes the integral of $f$,
\begin{equation}
I = \int_{-\pi}^\pi f(x) \,dx,
\end{equation}
the associated quadrature approximation
is the integral of $t_N^{}(x)$, which can be shown to be
equal to the result of applying the trapezoidal rule to $f$:
\begin{equation}
I_N^{} = h\!\sum_{k=-N}^N f(x_k^{}) = \int_{-\pi}^\pi t_N^{}(x) \,dx = 2\pi c_0^{} .
\label{traprule}
\end{equation}
It is known that if $f$ is continuous, then
\begin{equation}
\lim_{N\to\infty} | I - I_N^{} | = 0,
\label{quadconv}
\end{equation}
and if $f$ is Dini-continuous, for which H\"older or Lipschitz continuity
are sufficient conditions, then
\begin{equation}
\lim_{N\to\infty} \| f - t_N^{} \| = 0 .
\label{approxconv}
\end{equation}
Moreover, the convergence rates are tied to the
smoothness of $f$, with exponential convergence
if $f$ is analytic.  Here and throughout,
$\|\cdot\|$ is the maximum norm on $[-\pi,\pi)$.

The problem addressed in this paper is the generalization of
these results to configurations in which
the interpolation points are perturbed. For fixed
$\alpha\in ( 0,1/2)$, consider a set of points
\begin{equation}
\tilde x_k^{} = x_k^{} + s_k^{} h, \quad -N\le k \le N,
\quad |s_k|\le \alpha \, .
\label{ppoints}
\end{equation}
Note that since $\alpha < 1/2$, the $\tilde x_k^{}$ are necessarily
distinct.  Let $\tilde t_N^{}(x)$ be the unique degree-$N$
trigonometric interpolant to $\{f(\tilde x_k^{})\}$, and let
$\tilde I_N^{} = \int \tilde t_N^{}(x) \kern .7pt dx$ be the
corresponding quadrature approximation.  As in (\ref{traprule}),
this will be a linear combination of the function values, although
no longer with equal weights in general.

Let $\sigma>0$ be any positive real number, and write $\sigma
= \nu + \gamma$ with $\gamma\in (\kern .5pt 0,1]$.  We say
that $f$ {\em has $\sigma$ derivatives} if $f$ is $\nu$
times continuously differentiable and, moreover, $f^{(\nu)}$
is H\"older continuous with exponent $\gamma$.  Note that
if $\sigma$ is an integer, then for $f$ to ``have $\sigma$
derivatives'' means that $f$ is $\sigma-1$ times continuously
differentiable and $f^{(\sigma-1)}$ is Lipschitz continuous.
We will prove the following main theorem, whose central estimate
is the bound on $\|f-\tilde t_N^{}\|$ in (\ref{thmrate2}).
The estimates (\ref{thmconv})--(\ref{thmrate2}) are new, whereas
(\ref{thmrate3}) follows from the work of Kis~\cite{kis}, as
discussed in Section~\ref{confluent}.  Numerical illustrations
of these bounds can be found in~\cite{austin}.

\medskip

\begin{theorem}
\label{thm1}
For any $\alpha\in (\kern .5pt 0,1/2)$,
if $f$ is twice continuously differentiable, then
\begin{equation}
\lim_{N\to\infty} | I - \tilde I_N^{} | =
\lim_{N\to\infty} \| f - \tilde t_N^{} \| = 0.
\label{thmconv}
\end{equation}
More precisely, if
$f$ has $\sigma$ derivatives for some $\sigma > 4\alpha$, then
\begin{equation}
| I - \tilde I_N^{} | ,
\| f - \tilde t_N^{} \| = O(N^{4\alpha-\sigma}).
\label{thmrate2}
\end{equation}
If $f$ can be analytically continued
to a $2\pi$-periodic
function for $-a < \Im x < a$ for some $a>0$, then
for any\/ $\hat a < a$,
\begin{equation}
| I - \tilde I_N^{} | ,
\| f - \tilde t_N^{} \| = O(e^{-\hat aN}).
\label{thmrate3}
\end{equation}
\end{theorem}

\medskip

Our proofs are based on combining standard estimates of
approximation theory, the Jackson theorems, with a new bound
on the Lebesgue constants associated with perturbed grids,
Theorem~\ref{thmbound}.  Our bounds are close to sharp,
but not quite.  Based on extensive numerical experiments
presented in Section 3.3.2 of~\cite{austin}, we conjecture that
$4\alpha$ can be improved to $2\alpha$ in (\ref{thmrate2}) and
(\ref{bigestimate}); for (\ref{bigestimate}) the result would
probably then be sharp, but for (\ref{thmrate2}) a slight further
improvement may still be possible.  For the quadrature problem
in particular, further experiments presented in Section 3.5.2
of~\cite{austin} lead us to conjecture that $\tilde I_N^{}
\to I$ as $N\to\infty$ for all continuous functions $f$ for
all $\alpha<1/2$.  This conjecture is based on the theory of
P\'olya in 1933~\cite{polya}, who showed that such convergence
is ensured if and only if the sums of the absolute values of the
quadrature weights are bounded as $N\to\infty$.  Experiments
indicate that for all $\alpha<1/2$, these sums are indeed bounded
as required.  On the other hand, $\tilde I_N^{}\to I$ cannot be
guaranteed for any $\alpha\ge 1/2$, since in that case the
interpolation points may come together, making the quadrature
weights unbounded.

Theorems~\ref{thm1} and~\ref{thmbound} suggest that from the
point of view of approximation and quadrature, $\alpha = 1/4$
is not a special value.  In Section~\ref{sampling} we
comment on the significance of the appearance of this number
in the Kadec 1/4 theorem and more generally on the relationship
between approximation theory and sampling theory, two subjects
that address closely related questions and yet have little
overlap of literatures or experts.

All the estimates reported here were worked out by the first
author and presented in his D.\ Phil.\ thesis~\cite{austin}.
This work was motivated by work of the second author with
Weideman in the review article ``The exponentially convergent
trapezoidal rule''~\cite{tw}.  It is well known that on an
equispaced periodic grid, the trapezoidal rule is exponentially
convergent for periodic analytic integrands~\cite{davis,tw}.
With perturbed points, it seemed to us that exponential
convergence should still be expected, and we were surprised
to find that there seemed to be no literature on this subject.
A preliminary discussion was given in~\cite[Sec.~9]{tw}.

Section~\ref{lebesguesec} reduces Theorem~\ref{thm1} to a
bound on the Lebesgue constant, Theorem~\ref{thmbound}.
Sections~\ref{confluent} and \ref{sampling} are devoted
to comments on problems with $\alpha\ge 1/2$ and on
the link with sampling theory and Kadec's theorem,
respectively.  Section~\ref{bound} outlines the proof of
Theorem~\ref{thmbound}.

\section{\label{lebesguesec}Reduction to a Lebesgue constant estimate}

A fundamental tool of approximation theory is the Lebesgue
constant: for any linear projection $L:f\mapsto t$, the
Lebesgue constant is the operator norm $\Lambda = \|L\|$.
For our problem the operator is the map $\tilde L_N^{}$ from a
function $f$ to its trigonometric interpolant $\tilde t_N^{}$
through the values $\{f(\tilde{x}_k^{})\}$, and the norm on $L$
is the operator norm induced by $\|\cdot\|$, the $\infty$-norm
on $[-\pi,\pi)$.  We denote the Lebesgue constant by $\tilde
\Lambda_N$.

Lebesgue constants are linked to quality of approximations
by the following well-known bound.  If
$\tilde \Lambda_N^{}$ is the Lebesgue constant
associated with the projection $\tilde L_N^{}:
f \mapsto \tilde t_N^{}$ and $t_N^*$ is the best approximation to
$f$ of degree $N$, then
\begin{equation}
\|f-\tilde t_N^{}\| \le (1+\tilde \Lambda_N^{}) \| f-t_N^*\| .
\label{basicbound}
\end{equation}
It follows that if $\tilde \Lambda_N^{}$ is small, then $\tilde
t_N^{}$ is a near-optimal approximation to $f$.  If $f$ has
a certain smoothness property for which the optimal approximations
$t_N^*$ are known to converge at a certain rate, this
implies that the interpolants $\tilde t_N^{}$ converge at nearly
the same rate.

Applying (\ref{basicbound}), we prove Theorem~\ref{thm1}
by combining a bound on the Lebesgue constants $\tilde
\Lambda_N^{}$ with bounds on the best approximation errors $\|f-
t_N^*\|$.  Our estimates of best approximations are standard
Jackson theorems, going back to Dunham Jackson in 1911 and 1912.
The nonstandard part of the argument, which from a technical
point of view is the main contribution of this paper, is the
following estimate of the Lebesgue constant, the proof of which
is outlined in Section~\ref{bound}.

\medskip

\begin{theorem}
\label{thmbound}
There is a universal constant $C$ such that
\begin{equation}
\tilde \Lambda_N^{} \le {C(N^{4\alpha}-1)\over \alpha(1-2\alpha)}
\label{bigestimate}
\end{equation}
for all\/ $\alpha\in [\kern .5pt 0,1/2)$ and $N> 0$.
For $\alpha=0$ this bound is to be interpreted by its
limiting value given, for example, by L'Hopital's rule,
$\tilde\Lambda_N^{} \le 4\kern1pt C\log N$.
\end{theorem}

\medskip

The $\log N$ bound for an equispaced grid with $\alpha = 0$
is standard, so the substantive result here concerns $\alpha\in
(0,1/2)$.  This is what we can prove, but as mentioned in the
previous section, based on numerical experiments, we conjecture that
(\ref{bigestimate}) actually holds with $N^{4\alpha}$ replaced
by $N^{2\alpha}$.

Given Theorem~\ref{thmbound}, we prove Theorem~\ref{thm1} as follows.

\medskip

{\em Proof of Theorem\/~$\ref{thm1}$, given Theorem\/~$\ref{thmbound}$.}
The Jackson theorems of approximation theory relate the smoothness
of a function $f$ to the accuracy of its best
approximations~\cite{jackson,meinardus}.
According to one of these theorems, given for example as
Theorem~41 of~\cite{meinardus},
if $f$ is a periodic function
on $[-\pi,\pi)$ that has $\sigma$ derivatives for
some $\sigma>0$ in the sense defined in Section 1, then
\begin{equation}
\| f-t_N^*\|  = O(N^{-\sigma}).
\end{equation}
Combining this with Theorem~\ref{thmbound}
gives (\ref{thmrate2}).  The bound (\ref{thmrate3})
follows similarly from the estimate
\begin{equation}
\| f-t_N^*\|  = O(e^{-\hat aN})
\end{equation}
for any $2\pi$-periodic function $f$ analytic and bounded in
the strip of half-width $\hat a>0$ about the real axis; see, for
example, eq.~(7.17) of~\cite{tw}. \endproof

\section{\label{confluent}\boldmath $\alpha \ge 1/2$, confluent
points, and analytic functions} Our framework (\ref{ppoints}) for
perturbed points can be generalized to values $\alpha \ge 1/2$.
For $\alpha\in [1/2, 1)$, two grid points may coalesce, so one
must assume that $f'$ exists in order to ensure that there are
appropriate data to define an interpolation problem (in this case,
trigonometric Hermite interpolation).  Similarly for $\alpha\in
[1,3/2)$, three points may coalesce, so one must assume $f''$
exists; and so on analogously for any finite value of $\alpha$.
(We wrap grid points around as necessary if the perturbation moves
them outside of $[-\pi, \pi)$; equivalently, one could extend $f$
periodically.)

Looking at the statement of Theorem~\ref{thm1} but considering
values $\alpha \ge 1/2$, one notes that the assumption of
$\sigma>4\alpha$ derivatives is enough to ensure that the necessary
derivatives exist for the interpolation problem to make sense; the
conjectured sharper condition of $\sigma> 2\alpha$ derivatives is
also (\kern .7pt just) enough.  This coincidence seems suggestive,
and we consider it possible that Theorem~\ref{thm1} and its
conjectured improvement with $2\alpha$ may in fact be valid for
arbitrary $\alpha>0$, not just $\alpha \in (\kern .5pt 0,1/2)$.
We have not attempted to prove this, however.  As a practical
matter, trouble can be expected in floating-point arithmetic as
sample points coalesce, so we regard the case $\alpha \ge 1/2$
as somewhat theoretical.

Going further, what if we allow arbitrary perturbations of the
interpolation points, so that each $\tilde x_k^{}$ may lie
anywhere in $[-\pi,\pi)$?  Doing so makes sense mathematically
if $f$ is infinitely differentiable; so in particular, it
makes sense if $f$ is analytic, which implies that it can be
analytically continued to a $2\pi$-periodic function on the
whole real line.  We are now in an area of approximation theory
(and potential theory) going back to the work of Runge~\cite{runge}
and Fej\'er~\cite{fejer}, in which a major contributor was Joseph
Walsh~\cite{gaier,walsh}.  For arbitrary $x_k^{}$, convergence
will occur if $f$ is analytic in a sufficiently wide strip around
the real axis in the complex $x$-plane.  Repeated points
are permitted, with interpolation at such points interpreted
in the Hermite sense involving values of both the function
and its derivatives.  If the points $x_k$
are {\em uniformly distributed\/} in the sense that the fraction
of points falling in any interval $[a,b)\subseteq [-\pi,\pi)$
converges to $(b-a)/2\pi$ as $N\to\infty$, then it is enough for
$f$ to be analytic in {\em any\/} strip around the real axis.
Such results were first developed for polynomial approximation
on the unit circle of functions analytic in the unit disk, the
so-called Fej\'er--Kalm\'ar theorem~\cite{fejer,kalmar,walsh}.
The extension to functions analytic in an annulus was considered by
Hlawka~\cite{hlawka}, and the equivalent problem of trigonometric
interpolation of $2\pi$-periodic functions on $[-\pi,\pi)$ was
considered by Kis~\cite{kis}.  All these results may fail in
practice because of rounding errors on the computer, however.
For example, Figure~3.7 of~\cite{austin} shows an example with
uniformly distributed random interpolation points in $[-\pi,\pi)$,
with rounding errors beginning to take over at $N\approx 20$.
For the case of interpolation by algebraic polynomials, this
kind of effect is familiar in the context of the Runge phenomenon,
where polynomial interpolants in equispaced points in $[-1,1]$
will diverge on a computer as $N\to\infty$ even for a function
like $f(x) = \exp(x)$ for which in principle they should converge.

\section{\label{sampling}Sampling theory and the Kadec \boldmath$1/4$
theorem} The field of approximation theory
goes back to Borel, de la Vall\'ee Poussin, Fej\'er, Jackson,
Lebesgue, and others at the beginning of the 20th century, and
its central question might be characterized like this:

\medskip

\begin{quotation}
\noindent
\em Given a function $f$ of a certain regularity, how
fast do its approximations of a given kind converge?
\end{quotation}

\medskip

\noindent
For example, if $f$ is periodic and analytic on $[-\pi, \pi)$, then
its equispaced trigonometric interpolants converge exponentially.
The same holds if $f$ is analytic in a strip surrounding the
whole real line and satisfies a decay condition at $\infty$,
with trigonometric interpolants generalized to interpolatory
series of sinc functions.

The field of sampling theory goes back to Gabor, Kotelnikov,
Nyquist, Paley, Shannon, J. M. and E. T. Whittaker, and Wiener a
few years later.  Its central question might be characterized
like this:

\medskip

\begin{quotation}
\noindent
\em Given a function $f$ of a certain regularity, which of
its approximations of a given kind are
exactly equal to $f\kern 1pt ?$
\end{quotation}

\medskip

\noindent
For example, if $f$ is periodic and analytic on $[-\pi,\pi)$,
then its equispaced trigonometric interpolant is exact if $f$
is band-limited (has a Fourier series of compact support) and
the grid includes at least two points per wavelength for each
wave number present in the series.  The same holds if $f$ is a
band-limited analytic function on the whole real line, with the
Fourier series generalized to the Fourier transform, and again
with trigonometric interpolation generalized to sinc interpolation.

Obviously we have worded these characterizations to highlight
the similarities between the two fields, which in fact differ in
significant ways.  Still, it is remarkable how little interaction
there has been between the two.  What makes this relevant to the
present paper is that our theorems and orientation are very much
those of approximation theory, whereas most of the scientific
interest in perturbed grids in the past has been from the side
of sampling theory, and the Kadec 1/4 theorem is the known result
in this general area.

Kadec's theorem is an answer to a question of sampling theory that
originates with Paley and Wiener~\cite{pw}.  The exponentials
$\{\exp(i\lambda_k^{} x)\}$, $-\infty < k < \infty$, form an
orthonormal basis for $L^2[-\pi,\pi]$ if $\lambda_k^{} = k$
for each $k$.  Thus, the sampling theorist would say that one can
recover a function $f\in L^2[-\pi,\pi]$ from its inner products
with the functions $\{\exp(i\lambda_k^{} x)\}$.  Now suppose these
wave numbers are perturbed so that $|\tilde \lambda_k^{} - k | \le
\alpha$ for some fixed $\alpha$.  Can one still recover the signal?
Specifically, does the family $\{\exp(i\tilde \lambda_k^{} x)\}$
form a {\em Riesz basis} for $L^2[-\pi,\pi]$, that is, a basis
that is related to the original one by a bounded transformation
with a bounded inverse?  Paley and Wiener showed that this is
always the case for $\alpha < 1/\pi^2$, and Levinson showed it
is not always the case for $\alpha \ge 1/4$.  Kadec's theorem
shows that Levinson's construction was sharp: for any $\alpha <
1/4$, the family $\{\exp(i\tilde \lambda_k^{} x)\}$ forms a Riesz
basis~\cite{aldgro,christensen,kadec,young}.

Note that the standard setting of Kadec's theorem involves
perturbation of wave numbers from equispaced values, in contrast
to the results of this paper, which involve perturbation of
interpolation points from equispaced values.  In view of
the Fourier transform, however, these settings are related, so one might
imagine, based on Kadec's theorem, that $\alpha = 1/4$ might be a
critical value for trigonometric interpolation in perturbed points.
Instead, we have found that the critical value is $\alpha = 1/2$.

We explain this apparent discrepancy as follows.  The Paley--Wiener
theory and Kadec's theorem are results concerning the $L^2$
norm, which in many applications would represent energy.  In our
application of trigonometric interpolation, something related to
the $L^2$ norm does indeed happen at $\alpha = 1/4$.  Suppose we
look at a {\em $2$-norm Lebesgue constant} $\tilde\Lambda_N^{(2)}$
for the perturbed grid interpolation problem, defined as the
operator norm on $L: f\mapsto \tilde t_N^{}$ induced by the
discrete $\ell^2$-norm on the data $\{f(\tilde x_k^{})\}$ and
on the Fourier coefficients of the interpolant $\tilde t_k^{}$.
Numerical experiments reported in Section 3.4.3 of~\cite{austin}
indicate that whereas the usual $\infty$-norm Lebesgue constant
is unbounded for all $\alpha$, $\tilde\Lambda_N^{(2)}$ is bounded
as $N\to\infty$ for any $\alpha < 1/4$ but not always bounded for
$\alpha \ge 1/4$.  (Indeed Kadec's theorem may imply this result.)
For $\alpha \in (1/4,1/2)$, we conjecture
$\tilde \Lambda_N^{(2)} = O(N^{4\alpha - 1})$.

Thus a sampling theorist might say that for $\alpha \in
[1/4,1/2)$, trigonometric interpolation is {\em unstable} in
the sense that it may amplify signals unboundedly in $\ell^2$
as $N\to\infty$.  On the other hand the approximation theorist
might note that the instability is very weak, involving not
even one power of $N$.  Assuming that the conjectured sharpening of
the estimate (\ref{thmrate2}) of Theorem~\ref{thm1} is valid,
one derivative of smoothness of $f$ is enough to suppress the
instability, ensuring $\|f-\tilde t_N^{}\|\to 0$ as $N\to\infty$
for all $\alpha < 1/2$.  The numerical analyst might add that on a
computer, amplification of rounding errors by $o(N)$ is unlikely
to cause trouble.  For $\alpha \ge 1/2$, in strong contrast,
the amplification is unbounded in any norm even for finite $N$,
and trouble is definitely to be expected.

\section{\label{bound}Proof of the Lebesgue constant estimate,
Theorem~\ref{thmbound}} A full proof of Theorem~\ref{thmbound},
filling 20 pages, is the subject of Chapter~4 of the first author's
D.\ Phil.\ thesis~\cite{austin}.  Many detailed trigonometric estimates
are involved, and we do not know how to shorten it significantly.
For readers interested in full details, that chapter has been made
available in the Supplementary Materials attached to this paper.

Here, we outline the argument.
To prove the bound (\ref{bigestimate}) on the Lebesgue constant,
\begin{equation}
\tilde \Lambda_N^{} \le {C(N^{4\alpha}-1)\over \alpha(1-2\alpha)},
\label{bigestimateagain}
\end{equation}
we begin by noting that $\tilde \Lambda_N^{}$ is given by
\begin{equation}
\tilde \Lambda_N^{}= \max_{x\in[-\pi,\pi]} \tilde L(x),
\label{maxdef}
\end{equation}
where $\tilde L$ is the Lebesgue function
\begin{equation}
\tilde L(x) = \sum_{k=-N}^N |\tilde \ell_k^{}(x) |,
\label{basicsum}
\end{equation}
where $\tilde{\ell}_k^{}$ is the $k$th
Lagrange cardinal trigonometric polynomial for the perturbed grid,
\begin{equation}
\tilde \ell_k^{}(x) = \prod_{j\ne k}\left.
\sin\Big({x-\tilde x_j^{}\over 2}\Big)
\right/ \sin\Big({\tilde x_k^{}-\tilde x_j^{}\over 2}\Big).
\end{equation}
The function $\tilde\ell_k^{}(x)$ takes the values $1$ at $\tilde
x_k^{}$ and $0$ at the other grid points $\tilde x_j^{}$, and the
sum (\ref{basicsum}) adds up contributions at a point $x$ from
all the $2N+1$ cardinal functions associated with grid points to
its left and right.

The argument begins by showing that on the interval $[x_{-(k + 1)}^*,
x_{-k}^*]$, $\tilde{\ell}_0$ satisfies the bound
\begin{equation}
|\tilde{\ell}_0(x)| \leq M_k, \quad x\in [x_{-(k + 1)}^*, x_{-k}^*], \quad 0 \leq k \leq N,
\label{Mbounds}
\end{equation}
for certain numbers $M_0, \ldots, M_N$, independently of the choice of
perturbed points $\{\tilde{x}_k\}$.  The points $x_k^*$ are defined by $x_0^* =
0$, $x_{-(N + 1)}^* = -\pi$, and
\[
x_k^* = 2\arctan\left(\frac{\cos(kh) - \cos(\alpha h) + \tan(\tilde{x}_0/2)\sin(kh)}{\tan(\tilde{x}_0/2)\bigl(\cos(kh) + \cos(\alpha h)\bigr) - \sin(kh)}\right), \quad -N \leq k \leq N, k \neq 0;
\]
the most important fact about them is that they satisfy the inequalities
\[
(k - \alpha)h \leq x_k^* \leq (k + \alpha)h, \quad -N \leq k \leq N.
\]
Thus, (\ref{Mbounds}) bounds $\tilde{\ell}_0$ on certain subintervals of
$[-\pi, 0]$.  By exploiting symmetry, these bounds yield similar bounds on
$\tilde{\ell}_0$ on similar subintervals of $[0, \pi]$ as well as bounds on the
other $2N$ contributions to $\tilde{L}$ in (\ref{basicsum}).  We are eventually
led to the estimate
\[
\tilde{L}(x) \leq 9 \sum_{k = 0}^N M_k,
\]
which holds uniformly for $x \in [-\pi, \pi]$.

For sufficiently large $N$, the $M_k$ satisfy
\begin{equation}
M_k^{} \le {10 \pi\over 1-2\alpha}, \quad k = 0, 1
\label{Mineq1}
\end{equation}
and
\begin{equation}
M_k^{} \le {3 \pi (k+1)^{2\alpha}\over (1-2\alpha) (k-1)^{1-2\alpha}}, \quad 2 \le k \le N.
\label{Mineq2}
\end{equation}
The bound (\ref{bigestimateagain}) follows by an estimation of the sums of
(\ref{Mineq1}) and (\ref{Mineq2}) over all $k$.  The numbers $M_k^{}$ are
defined by
\begin{equation}
M_k^{} =
\max_{x \in [-\pi, 0] \cap R_k^{}} \frac{P_k^{}(x)}{Q_k^{}},
\quad 0 \le k \le N,
\label{Mdef}
\end{equation}
with
\begin{displaymath} P_k^{}(x)
= \prod_{i=1}^N \left|\sin\left(\frac{x - (i -
\alpha)h}{2}\right)\right|  \times\prod_{i=1}^k
\left|\sin\left(\frac{x + (i - \alpha)h}{2}\right)\right|
 \times\prod_{i = k + 1}^N \left|\sin\left(\frac{x + (i +
\alpha)h}{2}\right)\right|
\end{displaymath}
and
\begin{displaymath}
 Q_k^{} = \prod_{i = 1}^N \left|\sin\left(\frac{(2\alpha
- i)h}{2}\right)\right|  \times\prod_{i = 1}^k
\left|\sin\left(\frac{ih}{2}\right)\right|
\times\prod_{i = k + 1}^N \left|\sin\left(\frac{(2\alpha +
i)h}{2}\right)\right|.
\end{displaymath}
The set $R_k^{}$ in the definition of the range of
the maximum in (\ref{Mdef}) is the interval
\begin{displaymath}
R_k^{} = [(-k - 1 - \alpha)h, (-k + \alpha)h].
\end{displaymath}

\section*{Acknowledgements}

Many friends and colleagues have helped us along the way.  In
particular, we thank Stefan G\"uttel, Andrew Thompson, Alex Townsend,
and Kuan Xu.  Thompson showed us how we could improve $N^{4\alpha}$ to
$N^{4\alpha}-1$ in our main estimate (\ref{bigestimate}).

\vskip .1in

~~


\begin{thebibliography}{1}

\bibitem{DPhilThesis}
{\sc A.~P. Austin}, {\em Some New Results on and Applications of Interpolation
  in Numerical Computation}, {DPhil}\ thesis, Mathematical Institute,
  University of Oxford, 2016.

\bibitem{DR1984_2e}
{\sc P.~J. Davis and P.~Rabinowitz}, {\em Methods of Numerical Integration},
  Academic Press, New York, 2nd~ed., 1984.

\bibitem{GR2007_7e}
{\sc I.~S. Gradshteyn and I.~M. Ryzhik}, {\em Table of Integrals, Series, and
  Products}, Elsevier, Burlington, MA, 7th~ed., 2007.

\end{thebibliography}


\begin{thebibliography}{1}

\bibitem{austin}
{\sc A. P. Austin,}
{\em Some New Results on and Applications of
Interpolation in Numerical Computation,}
D.\ Phil.\ thesis, Mathematical Institute,
University of Oxford, 2016.

\bibitem{aldgro}
{\sc A. Aldroubi and K. Gr\"ochenig,}
{\em Nonuniform sampling and reconstruction in
shift-invariant spaces,} SIAM Rev., 43 (2001),
pp.~585--620.

\bibitem{christensen}
{\sc O. Christensen,} {\em An Introduction to
Frames and Riesz Bases,} Birkh\"auser, 2003.

\bibitem {davis}
{\sc P. J. Davis,}
{\em On the numerical integration of periodic analytic
functions,} in R. E. Langer, ed., On Numerical Integration:
Proceedings of a Symposium, Madison, April 21--23, 1958,
Math.\ Res.\ Ctr., U. of Wisconsin, 1959, pp.~45--59.

\bibitem{fejer}
{\sc L. Fej\'er,} {\em Interpolation und konforme
Abbildung,} G\"ott.\ Nachr.\ (1918), pp.~319--331.

\bibitem{gaier}
{\sc D. Gaier,}
{\em Lectures on Complex Approximation,}
Birkh\"auser, 1987.

\bibitem{hlawka}
{\sc E. Hlawka,}
{\em Interpolation analytischer Funktionen auf dem Einheitskreis,}
in P. Tur\'an, ed., Number Theory and Analysis, Plenum, New York, 1969,
pp.~99--118.

\bibitem{jackson}
{\sc D. Jackson,} {\em The Theory of Approximation,}
Amer.\ Math.\ Soc., 1930.

\bibitem{kadec}
{\sc M. I. Kadec,} {\em The exact value
of the Paley--Wiener constant,} Sov.\ Math.\ Dokl., 5 (1964), 559--561.

\bibitem{kalmar}
{\sc L. Kalm\'ar,}
{\em On interpolation,} Mat.\ Lapok, 33 (1926), pp.~120--149 (Hungarian).

\bibitem{kis}
{\sc O. Kis,}
{\em On the convergence of the trigonometrical and harmonical interpolation,}
Acta Math.\ Acad.\ Sci.\ Hung., 7 (1956), pp.~173--200 (Russian).

\bibitem{meinardus}
{\sc G. Meinardus,} {\em Approximation of Functions:
Theory and Numerical Methods,} Springer, 1967.

\bibitem{pw}
{\sc R. E. A. C. Paley and N. Wiener,}
{\em Fourier Transforms in the Complex Domain,}
Amer.\ Math.\ Soc., 1934.

\bibitem{polya}
{\sc G. P\'olya,}
{\em \"Uber die Konvergenz von Quadraturverfahren,}
Math.\ Z., 37 (1933), pp.~264--286.

\bibitem{runge}
{\sc C. Runge,}
{\em \"Uber empirische Funktionen und die
Interpolation zwischen \"aquidistanten Ordinaten,} Z. Math.\ Phys.,
46 (1901), pp.~224--243.

\bibitem{tw}
{\sc L. N. Trefethen and J. A. C. Weideman,}
{\em The exponentially convergent trapezoidal rule,}
SIAM Rev., 56 (2014), pp.~385--458.

\bibitem{walsh}
{\sc J. L. Walsh,} {\em Interpolation and Approximation
by Rational Functions in the Complex Domain,} Amer.\ Math.\ Soc.,
1935.

\bibitem{young}
{\sc R. M. Young,} {\em An Introduction to
Non-Harmonic Fourier Series,} revised ed., Academic Press, 2001.

\end{thebibliography}
\end{document}


\maketitle

\thispagestyle{plain}


In this supplementary appendix, we reproduce the proof of Theorem~2.1 given in
\cite[Ch.\ 4]{DPhilThesis} with a few minor modifications to make the text
self-contained.  For ease of reference, we repeat our basic notational setup
from the main article here.  If $K = 2N + 1$ is an odd integer, the
zero-centered equispaced grid of length $K$ in $[-\pi, \pi)$ consists of the
points
\begin{equation}\label{EQN:EquiPts}
x_k = kh, \qquad -N \leq k \leq N,
\end{equation}
where $h = 2\pi/K$ is the grid spacing.  We consider the perturbed grid
\begin{equation}\label{EQN:PertPts}
\tilde{x}_k = x_k + s_k h, \qquad |s_k| \leq \alpha,
\end{equation}
where the parameter $\alpha$ is a fixed value in the range $0 \leq \alpha <
1/2$.  The $k$th trigonometric Lagrange basis function associated with the
perturbed grid is denoted by $\tilde{\ell}_k$, i.e.
\[
\tilde{\ell}_k(x) = \prod_{\substack{j = -N \\ j \neq k}}^N \frac{\sin\left(\frac{x - \tilde{x}_j}{2}\right)}{\sin\left(\frac{\tilde{x}_k - \tilde{x}_j}{2}\right)}.
\]
We have $\tilde{\ell}_k(x_j) = 1$ if $j = k$ and $0$ if $j \neq k$.  From
(5.2) and (5.3) in the main article, we have
\begin{equation}\label{EQN:PertLebConst}
\tilde{\Lambda}_N = \max_{x \in [-\pi, \pi]} \sum_{k = -N}^N |\tilde{\ell}_k(x)|.
\end{equation}

Our argument can be loosely outlined as follows.  The bulk of the work is
devoted to bounding $|\tilde{\ell}_0(x)|$, which takes several steps to
accomplish.  Taking $x$ as fixed, we determine the choice of the points
$\tilde{x}_j$ that maximizes $|\tilde{\ell}_0(x)|$ and then bound the maximum
using integrals.  Since the resulting bound is independent of $\tilde{x}_j$, we
can exploit symmetry to obtain bounds on $|\tilde{\ell}_k(x)|$ for $k \neq 0$.
We then sum these bounds over $k$ to obtain a bound on $\tilde{\Lambda}_N$.

We begin with the following result, which shows that to bound
$|\tilde{\ell}_0(x)|$ we need consider only grids in which all the points,
possibly aside from $\tilde{x}_0$, are perturbed by the maximum amount of
$\alpha h$.

\medskip

\begin{lemma}
For all $x \in [-\pi, \pi]$ and $-N \leq j \leq N$, $j \neq 0$,
\begin{equation}\label{EQN:l0FactorBound}
\left|\frac{\sin\left(\frac{x - \tilde{x}_j}{2}\right)}{\sin\left(\frac{\tilde{x}_0 - \tilde{x}_j}{2}\right)}\right| \leq \max\left(\left|\frac{\sin\left(\frac{x - (j - \alpha)h}{2}\right)}{\sin\left(\frac{\tilde{x}_0 - (j - \alpha)h}{2}\right)}\right|, \left|\frac{\sin\left(\frac{x - (j + \alpha)h}{2}\right)}{\sin\left(\frac{\tilde{x}_0 - (j + \alpha)h}{2}\right)}\right|\right).
\end{equation}
\end{lemma}

\begin{proof}
The statement is trivially true if $x = \tilde{x}_0$.  If $x \neq \tilde{x}_0$,
then from
\[
\ldiff{}{t} \frac{\sin\left(\frac{x - t}{2}\right)}{\sin\left(\frac{\tilde{x}_0 - t}{2}\right)} = \frac{1}{2} \frac{\sin\left(\frac{x - \tilde{x}_0}{2}\right)}{\left[\sin\left(\frac{\tilde{x}_0 - t}{2}\right)\right]^2},
\]
we see that $t \mapsto \sin\bigl((x - t)/2\bigr)/\sin\bigl((\tilde{x}_0 -
t)/2\bigr)$ has no critical points in $[-\pi, \pi]$ apart from $t =
\tilde{x}_0$, where it is singular.  In particular, it has no critical points
in any of the intervals $[(j - \alpha)h, (j + \alpha)h]$ for $-N \leq j \leq
N$, $j \neq 0$, and therefore must assume its extreme values on these intervals
at the endpoints.  Since $\tilde{x}_j \in [(j - \alpha)h, (j + \alpha)h]$ for
each $j$, we are done.
\end{proof}

\medskip

Which of the two arguments to the maximum function on the right-hand side of
\eqref{EQN:l0FactorBound} is larger depends on both $x$ and $j$.  We need to
understand the exact conditions under which each one takes over.  Our first
step in this direction is the following lemma, which tells us when the two are
equal.

\medskip

\begin{lemma}\label{LEM:xjs}
For $0 < \alpha < 1/2$ and $-N \leq j \leq N$, $j \neq 0$, the equation
\begin{equation}\label{EQN:TakeoverEqn}
\left|\frac{\sin\left(\frac{x - (j - \alpha)h}{2}\right)}{\sin\left(\frac{\tilde{x}_0 - (j - \alpha)h}{2}\right)}\right| = \left|\frac{\sin\left(\frac{x - (j + \alpha)h}{2}\right)}{\sin\left(\frac{\tilde{x}_0 - (j + \alpha)h}{2}\right)}\right|
\end{equation}
has exactly two solutions in $[-\pi, \pi]$:  $x = \tilde{x}_0$ and $x = x_j^*$,
where%
\footnote{Here, $\arctan$ denotes the principal branch of the inverse tangent
function.}%
\[
x_j^* = 2 \arctan\left(\frac{\cos(jh) - \cos(\alpha h) + \tan\bigl(\tilde{x}_0/2\bigr)\sin(jh)}{\tan\bigl(\tilde{x}_0/2\bigr)\bigl(\cos(jh) + \cos(\alpha h)\bigr) - \sin(jh)}\right).
\]
\end{lemma}

\begin{proof}
Multiplying through by the denominators of both sides and applying some
trigonometric identities, we find that \eqref{EQN:TakeoverEqn} can be
reduced to
\begin{multline}\label{EQN:MaxEqn}
\left|\cos\left(\frac{\tilde{x}_0 - x}{2} + \alpha h\right) - \cos\left(\frac{\tilde{x}_0 + x}{2} - jh\right)\right| \\ = \left|\cos\left(\frac{\tilde{x}_0 - x}{2} - \alpha h\right) - \cos\left(\frac{\tilde{x}_0 + x}{2} - jh\right)\right|.
\end{multline}
If the expressions within the absolute value signs on either side of
\eqref{EQN:MaxEqn} are equal, then we have
\[
\cos\left(\frac{\tilde{x}_0 - x}{2} + \alpha h\right) = \cos\left(\frac{\tilde{x}_0 - x}{2} - \alpha h\right).
\]
In order to solve this equation, we consider two cases.

\emph{\textit{Case 1}:  $(\tilde{x}_0 - x)/2 + \alpha h = (\tilde{x}_0 -
x)/2 - \alpha h + 2n\pi$ for some integer $n$.}  Rearranging gives $\alpha h =
n\pi$, and substituting for $h$, we arrive at $\alpha = Kn/2$.  Since $\alpha <
1/2$, this can hold only if $n = 0$, in which case $\alpha = 0$, but this is
disallowed by our hypotheses.

\emph{\textit{Case 2}: $(\tilde{x}_0 - x)/2 + \alpha h = \alpha h -
(\tilde{x}_0 - x)/2 + 2n\pi$ for some integer $n$.}  If this holds, then
$\tilde{x}_0 - x = 4n\pi$, but this can happen only if $n = 0$, since
$\tilde{x}_0 - x \in [-\pi - \alpha h, \pi + \alpha h]$, and this interval is
contained in $[-2\pi, 2\pi]$ because $\alpha h \leq \pi$.  Thus, $x =
\tilde{x}_0$.

We conclude that $x = \tilde{x}_0$ is the only solution when the expressions
within the absolute value signs on either side of \eqref{EQN:MaxEqn} are
equal.  On the other hand, if they are equal but of opposite sign, we get
\[
2\cos\left(\frac{\tilde{x}_0 + x}{2} - jh\right) = \cos\left(\frac{\tilde{x}_0 - x}{2} + \alpha h\right) + \cos\left(\frac{\tilde{x}_0 - x}{2} - \alpha h\right).
\]
Simplifying the right-hand side to $2\cos\bigl((\tilde{x}_0 -
x)/2\bigr)\cos(\alpha h)$ and then expanding both sides out completely using
trigonometric identities, we find that
\begin{multline*}
\cos\left(\frac{\tilde{x}_0}{2}\right)\cos\left(\frac{x}{2}\right)\cos(jh) - \sin\left(\frac{\tilde{x}_0}{2}\right)\sin\left(\frac{x}{2}\right)\cos(jh) \\ + \sin\left(\frac{\tilde{x}_0}{2}\right)\cos\left(\frac{x}{2}\right)\sin(jh) + \cos\left(\frac{\tilde{x}_0}{2}\right)\sin\left(\frac{x}{2}\right)\sin(jh) \\ = \cos\left(\frac{\tilde{x}_0}{2}\right)\cos\left(\frac{x}{2}\right)\cos(\alpha h) +\sin\left(\frac{\tilde{x}_0}{2}\right)\sin\left(\frac{x}{2}\right)\cos(\alpha h).
\end{multline*}
Dividing both sides of this through by $\cos(\tilde{x}_0/2)\cos(x/2)$ and
rearranging, we obtain
\[
\tan\left(\frac{x}{2}\right) = \frac{\cos(jh) - \cos(\alpha h) + \tan\bigl(\tilde{x}_0/2\bigr)\sin(jh)}{\tan\bigl(\tilde{x}_0/2\bigr)\bigl(\cos(jh) + \cos(\alpha h)\bigr) - \sin(jh)}.
\]
Taking the inverse tangent of both sides and multiplying by $2$, we arrive at
$x = x_j^*$.
\end{proof}

\medskip

To move forward, we need a better understanding of the locations of the points
$x_j^*$.  The requisite inequalities are simple to state and are given in
Lemma~\ref{LEM:xjsBounds}, but first we pause to establish a minor fact
that we will need in their proof.

\medskip

\begin{lemma}\label{LEM:sinjhVSsinah}
For $|t| \leq \alpha$ and $-N \leq j \leq N$, $j \neq 0$,
\[
\bigl|\sin\bigl((j + t)h\bigr)\bigr| > \sin\bigl((\alpha - |t|)h\bigr).
\]
\end{lemma}

\begin{proof}
This is a consequence of the following chain of inequalities:
\begin{multline*}
0 \leq (\alpha - |t|)h < (1 - |t|)h \leq (|j| - |t|)h \\
\leq (|j| + |t|)h \leq (N + |t|)h < \left(N + |t| + \frac{1}{2} - \alpha\right)h = \pi - (\alpha - |t|)h \leq \pi.
\end{multline*}
\end{proof}

\medskip

\begin{lemma}\label{LEM:xjsBounds}
For $-N \leq j \leq N$, $j \neq 0$, and $0 < \alpha < 1/2$,
\[
(j - \alpha)h < x_j^* < (j + \alpha)h.
\]
\end{lemma}

\begin{proof}
Let
\[
f(t) = \frac{\cos(jh) - \cos(\alpha h) + t\sin(jh)}{t\bigl(\cos(jh) + \cos(\alpha h)\bigr) - \sin(jh)}.
\]
Note that $f\bigl(\tan(\tilde{x}_0/2)\bigr) = \tan(x_j^*/2)$.  A
straightforward computation using the quotient rule and some trigonometric
identities shows that
\[
f'(t) = - \left(\frac{\sin(\alpha h)}{t\bigl(\cos(jh) + \cos(\alpha h)\bigr) - \sin(jh)}\right)^2,
\]
which is always negative wherever it is defined.  By
Lemma~\ref{LEM:sinjhVSsinah}, we have $|\sin(jh)| > \sin(\alpha h)$ for
each $j$.  Furthermore, note that $j \neq 0$ implies $N \geq 1$, so that
$\alpha h < \pi/3$, and so $\cos(\alpha h) > 0$.  Therefore, $|\cos(jh) +
\cos(\alpha h)| \leq 1 + \cos(\alpha h)$ for each $j$.  It follows that
\[
\left|\frac{\sin(jh)}{\cos(jh) + \cos(\alpha h)}\right| > \frac{\sin(\alpha h)}{1 + \cos(\alpha h)} = \tan\left(\frac{\alpha h}{2}\right).
\]
Hence, the singularity in $f$ is outside the interval $[\tan(-\alpha h/2),
\tan(\alpha h/2)]$, and we conclude that
\[
f\bigl(\tan(\alpha h/2)\bigr) \leq f\bigl(\tan(\tilde{x}_0/2)\bigr) \leq f\bigl(\tan(-\alpha h/2)\bigr).
\]

Next, consider the function $g_+$ and the number $M_+$ defined by
\begin{align*}
g_+(t) &= \frac{\cos(jh) - t + \tan(-\alpha h/2)\sin(jh)}{\tan(-\alpha h/2)\bigl(\cos(jh) + t\bigr) - \sin(jh)}, \\
M_+ &= \frac{\cos(jh) - 1 + \tan(-\alpha h/2)\sin(jh)}{\tan(-\alpha h/2)\bigl(\cos(jh) - 1\bigr) - \sin(jh)}.
\end{align*}
Note that $g_+\bigl(\cos(\alpha h)\bigr) = f\bigl(\tan(-\alpha h/2)\bigr)$ and
that
\[
M_+ = \frac{\frac{\cos(jh) - 1}{\sin(jh)} + \tan(-\alpha h/2)}{\tan(-\alpha h/2)\frac{\cos(jh) - 1}{\sin(jh)} - 1} = \frac{\tan(jh/2) - \tan(-\alpha h/2)}{1 + \tan(jh/2)\tan(-\alpha h/2)} = \tan\left(\frac{(j + \alpha)h}{2}\right).
\]
Therefore, if we can show that $g_+\bigl(\cos(\alpha h)\bigr) < M_+$, we will
have that $\tan(x_j^*/2) < \tan\bigl((j + \alpha)h/2)$, which implies that
$x_j^* < (j + \alpha)h$, as desired.  The remainder of the proof will be
devoted to establishing this fact.  The lower bound on $x_j^*$ can be derived
by considering the function $g_-$ and the number $M_-$ defined by
\begin{align*}
g_-(t) &= \frac{\cos(jh) - t + \tan(\alpha h/2)\sin(jh)}{\tan(\alpha h/2)\bigl(\cos(jh) + t\bigr) - \sin(jh)}, \\
M_- &= \frac{\cos(jh) - 1 + \tan(\alpha h/2)\sin(jh)}{\tan(\alpha h/2)\bigl(\cos(jh) - 1\bigr) - \sin(jh)}
\end{align*}
and arguing similarly.  We omit the details.

To show that $g_+\bigl(\cos(\alpha h)\bigr) < M_+$, we begin by noting that by
multiplying the numerator and denominator of both $g_+(t)$ and $M_+$ by
$\cos(-\alpha h/2)$ and applying some trigonometric identities, they can be
rewritten as
\[
g_+(t) = \frac{\cos(\alpha h/2) t - \cos\bigl((j + \alpha/2)h\bigr)}{\sin(\alpha h/2)t + \sin\bigl((j + \alpha/2)h\bigr)}, \qquad M_+ = -\frac{\cos(\alpha h/2) - \cos\bigl((j + \alpha/2)h\bigr)}{\sin(\alpha h/2) - \sin\bigl((j + \alpha/2)h\bigr)}.
\]
Consider the affine function $\varphi$ obtained by multiplying together the
denominators in these new expressions for $g_+$ and $M_+$, where that of the
latter is taken to include the leading minus sign:
\begin{multline*}
\varphi(t) = -\sin(\alpha h/2)\Bigl(\sin(\alpha h/2) - \sin\bigl((j + \alpha/2)h\bigr)\Bigr)t \\
- \sin\bigl((j + \alpha/2)h\bigr)\Bigl(\sin(\alpha h/2) - \sin\bigl((j + \alpha/2)h\bigr)\Bigr).
\end{multline*}
We will show that $\varphi\bigl(\cos(\alpha h)\bigr) > 0$.  First, note that
$\varphi(t) = 0$ at $t = t_0 = -\sin\bigl((j + \alpha/2)h\bigr)/\sin(\alpha
h/2)$ and that by Lemma~\ref{LEM:sinjhVSsinah}, this point lies outside
of the interval $[-1, 1]$.  Next, observe that $\sin(\alpha h/2) > 0$, that
$\sin\bigl((j + \alpha/2)h\bigr)$ has the same sign as $j$, and that
\[
\varphi'(t) = -\sin(\alpha h/2)\Bigl(\sin(\alpha h/2) - \sin\bigl((j + \alpha/2)h\bigr)\Bigr).
\]
If $j < 0$, then $\sin(\alpha h/2) - \sin\bigl((j + \alpha/2)h\bigr) > 0$
trivially, so $\varphi'(t) < 0$.  Thus, $\varphi(t) > 0$ for $t < t_0$.
Inspecting the formula for $t_0$, we find that $t_0 > 0$ in this case.  Since
$t_0$ cannot lie in the interval $[-1, 1]$, it must further be true that $t_0 >
1$.  As $\cos(\alpha h) \leq 1$, we have that $\cos(\alpha h) < t_0$, as
desired.  On the other hand, if $j > 0$, then $\sin(\alpha h/2) - \sin\bigl((j
+ \alpha/2)h\bigr) < 0$ by Lemma~\ref{LEM:sinjhVSsinah}, and we have that
$\varphi'(t) > 0$, so that $\varphi(t) > 0$ for $t > t_0$.  But $t_0 < 0$ in
this case, and since $\cos(\alpha h) > 0$, we have $\cos(\alpha h) > t_0$, and
we are done.

It follows that $g_+\bigl(\cos(\alpha h)\bigr) < M_+$ is equivalent to the
inequality
\begin{multline*}
-\Bigl(\cos(\alpha h/2)\cos(\alpha h) - \cos\bigl((j + \alpha/2)h\bigr)\Bigr) \Bigl(\sin(\alpha h/2) - \sin\bigl((j + \alpha/2)h\bigr)\Bigr) \\
< \Bigl(\sin(\alpha h/2)\cos(\alpha h) + \sin\bigl((j + \alpha/2)h\bigr)\Bigr) \Bigl(\cos(\alpha h/2) - \cos\bigl((j + \alpha/2)h\bigr)\Bigr).
\end{multline*}
Expanding out the products, moving all terms involving $\cos(\alpha h)$ to the
left and those not involving it to the right, and using some trigonometric
identities to simplify the result, we find that this in turn is equivalent to
\[
\Bigl(\sin\bigl((j + \alpha)h\bigr) - \sin(\alpha h)\Bigr)\cos(\alpha h) < \sin(jh).
\]
Next, we expand $\sin\bigl((j + \alpha)h\bigr)$ and move all terms involving
$\sin(jh)$ to the right, leaving us with
\[
\bigl(\cos(jh) - 1\bigr)\sin(\alpha h)\cos(\alpha h) < \sin(jh)\bigl(1 - \bigl(\cos(\alpha h)\bigr)^2\bigr).
\]
Using the identities $1 - \cos(jh) = \sin(jh)\tan(jh/2)$ and $1 -
\bigl(\cos(\alpha h)\bigr)^2 = \sin\bigl(\alpha h\bigr)^2$, we can rearrange
this one more time to find that our original inequality is equivalent to
\[
\sin(jh)\bigl(\tan(\alpha h) + \tan(jh/2)\bigr) > 0.
\]
If $j > 0$, then since $\sin(jh) > 0$, this is equivalent to $-\tan(jh/2) <
\tan(\alpha h)$, which holds trivially, since the left-hand side is negative,
while the right-hand side is positive.  If $j < 0$, then $\sin(jh) < 0$, and
the inequality is equivalent to $-\tan(jh/2) > \tan(\alpha h)$.  Taking
inverse tangents, we see that this is equivalent to $\alpha < -j/2$, and
this inequality holds, since $-j \geq 1$ and $\alpha < 1/2$.  This completes
the proof.
\end{proof}

\medskip

Assembling these results, we can prove the following statement about the
right-hand side of \eqref{EQN:l0FactorBound}.

\medskip

\begin{lemma}\label{LEM:l0FactorBoundMaxEval}
We have
\[
\max\left(\left|\frac{\sin\left(\frac{x - (j - \alpha)h}{2}\right)}{\sin\left(\frac{\tilde{x}_0 - (j - \alpha)h}{2}\right)}\right|, \left|\frac{\sin\left(\frac{x - (j + \alpha)h}{2}\right)}{\sin\left(\frac{\tilde{x}_0 - (j + \alpha)h}{2}\right)}\right|\right) = \left|\frac{\sin\left(\frac{x - (j - \alpha)h}{2}\right)}{\sin\left(\frac{\tilde{x}_0 - (j - \alpha)h}{2}\right)}\right|
\]
when $1 \leq j \leq N$ and $x \in [-\pi, \tilde{x}_0] \cup [x_j^*, \pi]$ or
when $-N \leq j \leq -1$ and $x \in [x_j^*, \tilde{x}_0]$, and
\[
\max\left(\left|\frac{\sin\left(\frac{x - (j - \alpha)h}{2}\right)}{\sin\left(\frac{\tilde{x}_0 - (j - \alpha)h}{2}\right)}\right|, \left|\frac{\sin\left(\frac{x - (j + \alpha)h}{2}\right)}{\sin\left(\frac{\tilde{x}_0 - (j + \alpha)h}{2}\right)}\right|\right) = \left|\frac{\sin\left(\frac{x - (j + \alpha)h}{2}\right)}{\sin\left(\frac{\tilde{x}_0 - (j + \alpha)h}{2}\right)}\right|
\]
when $1 \leq j \leq N$ and $x \in [\tilde{x}_0, x_j^*]$ or when $-N \leq j \leq
-1$ and $x \in [-\pi, x_j^*] \cup [\tilde{x}_0, \pi]$.
\end{lemma}

\medskip

\begin{proof}
We will give the proof assuming $1 \leq j \leq N$; the proof for $-N \leq j
\leq -1$ is similar.  When $\alpha = 0$, there is nothing to prove, so we may
assume $\alpha > 0$.  By Lemma~\ref{LEM:xjs}, the two arguments of the
maximum function are equal only at $x = \tilde{x}_0$ and $x = x_j^*$, and by
Lemma~\ref{LEM:xjsBounds}, we have $-\pi < \tilde{x}_0 < (j - \alpha)h <
x_j^* < (j + \alpha)h < \pi$.  Evaluating both arguments of the maximum
function at $x = (j - \alpha)h$, we see that the first is zero, while the
second is nonzero.  Thus, the second must be the larger on $[\tilde{x}_0,
x_j^*]$.  Evaluating at $x = (j + \alpha)h$, the situation is reversed, and by
periodicity we find that the first must be the larger on $[-\pi, \tilde{x}_0]
\cup [x_j^*, \pi]$.
\end{proof}

\medskip

This lemma is all we need for maximizing the factors in $|\tilde{\ell}_0(x)|$
with respect to the $\tilde{x}_j$ for $j \neq 0$.  We would like to do
something similar for $\tilde{x}_0$.  Unfortunately, the dependence on
$\tilde{x}_0$ of the various cases in this result tells us that we cannot go
further and maximize any one factor over $\tilde{x}_0$ independently of $x$.
The next result shows that we can get around this by pairing up the factors at
$\pm j$ for $1 \leq j \leq N$ instead of considering them in isolation.

Note that we state the result only for $x \in [-\pi, 0]$.  The reason is that,
by symmetry, any bound we obtain on $|\tilde{\ell}_0(x)|$ for $x \in [-\pi, 0]$
that is independent of $x$ must also hold for $x \in [0, \pi]$.  We will
therefore ignore the case of $x \in [0, \pi]$ until we reach the end of our
argument, at which point we will see that it has been taken care of for free.
Alternatively, one could write out an analogous argument that assumes $x \in
[0, \pi]$ instead.

\medskip

\begin{lemma}\label{LEM:MaxOverx0t}
For $x \in [-\pi, 0]$ and $1 \leq j \leq N$,
\[
\left|\frac{\sin\left(\frac{x - \tilde{x}_{-j}}{2}\right)\sin\left(\frac{x - \tilde{x}_j}{2}\right)}{\sin\left(\frac{\tilde{x}_0 - \tilde{x}_{-j}}{2}\right)\sin\left(\frac{\tilde{x}_0 - \tilde{x}_j}{2}\right)}\right| \leq \begin{cases}
\displaystyle\left|\frac{\sin\left(\frac{x + (j - \alpha)h}{2}\right)\sin\left(\frac{x - (j - \alpha)h}{2}\right)}{\sin\left(\frac{jh}{2}\right)\sin\left(\frac{(2\alpha - j)h}{2}\right)}\right| & -\pi \leq x \leq x_{-j}^* \\[2 em]
\displaystyle\left|\frac{\sin\left(\frac{x + (j + \alpha)h}{2}\right)\sin\left(\frac{x - (j - \alpha)h}{2}\right)}{\sin\left(\frac{(2\alpha + j)h}{2}\right)\sin\left(\frac{(2\alpha - j)h}{2}\right)}\right| & x_{-j}^* \leq x \leq 0.
\end{cases}
\]
\end{lemma}

\begin{proof}
Fix $x$, and define the functions $f_1$, $f_2$, and $f_3$ by
\begin{align*}
f_1(t) &= \frac{\sin\left(\frac{x + (j - \alpha)h}{2}\right)\sin\left(\frac{x - (j - \alpha)h}{2}\right)}{\sin\left(\frac{t + (j - \alpha)h}{2}\right)\sin\left(\frac{t - (j - \alpha)h}{2}\right)} = \frac{\cos\bigl((j - \alpha)h\bigr) - \cos(x)}{\cos\bigl((j - \alpha)h\bigr) - \cos(t)} \\
f_2(t) &= \frac{\sin\left(\frac{x + (j + \alpha)h}{2}\right)\sin\left(\frac{x - (j - \alpha)h}{2}\right)}{\sin\left(\frac{t + (j + \alpha)h}{2}\right)\sin\left(\frac{t - (j - \alpha)h}{2}\right)} = \frac{\cos(jh) - \cos(x + \alpha h)}{\cos(jh) - \cos(t + \alpha h)} \\
f_3(t) &= \frac{\sin\left(\frac{x + (j - \alpha)h}{2}\right)\sin\left(\frac{x - (j + \alpha)h}{2}\right)}{\sin\left(\frac{t + (j - \alpha)h}{2}\right)\sin\left(\frac{t - (j + \alpha)h}{2}\right)} = \frac{\cos(jh) - \cos(x - \alpha h)}{\cos(jh) - \cos(t - \alpha h)}.
\end{align*}
Note that only the denominators of these functions vary with $t$; the
numerators are constant.  By Lemma~\ref{LEM:l0FactorBoundMaxEval}, we
have
\begin{equation}\label{EQN:OnePairUnmaxBound}
\left|\frac{\sin\left(\frac{x - \tilde{x}_{-j}}{2}\right)\sin\left(\frac{x - \tilde{x}_j}{2}\right)}{\sin\left(\frac{\tilde{x}_0 - \tilde{x}_{-j}}{2}\right)\sin\left(\frac{\tilde{x}_0 - \tilde{x}_j}{2}\right)}\right| \leq \begin{cases}
|f_1(\tilde{x}_0)| & -\pi \leq x \leq x_{-j}^* \\
|f_2(\tilde{x}_0)| & x_{-j}^* \leq x \leq \tilde{x}_0 \\
|f_3(\tilde{x}_0)| & \tilde{x}_0 \leq x \leq x_j^*.
\end{cases}
\end{equation}
Recalling that $\tilde{x}_0 \in [-\alpha h, \alpha h]$, by maximizing
$|f_1(t)|$, $|f_2(t)|$, and $|f_3(t)|$ over $t \in [-\alpha h, \alpha h]$ under
the appropriate conditions on $x$, we will show that this inequality may be
replaced by
\[
\left|\frac{\sin\left(\frac{x - \tilde{x}_{-j}}{2}\right)\sin\left(\frac{x - \tilde{x}_j}{2}\right)}{\sin\left(\frac{\tilde{x}_0 - \tilde{x}_{-j}}{2}\right)\sin\left(\frac{\tilde{x}_0 - \tilde{x}_j}{2}\right)}\right| \leq \begin{cases}
|f_1(\alpha h)| & -\pi \leq x \leq x_{-j}^* \\
|f_2(\alpha h)| & x_{-j}^* \leq x \leq 0,
\end{cases}
\]
and this is the inequality we are trying to establish.  We consider three
cases.

\emph{\textit{Case 1}:  $-\pi \leq x \leq x_{-j}^*$.}  In this case, the
right-hand side of \eqref{EQN:OnePairUnmaxBound} is governed by $f_1$.
The denominator of $f_1$ has a critical point in $[-\alpha h, \alpha h]$ at $t
= 0$, and it takes on identical values at the endpoints $\pm \alpha h$.  Since
\[
0 < \alpha h < (1 - \alpha)h \leq (j - \alpha)h \leq (N - \alpha)h < \left(N + \frac{1}{2}\right)h = \pi,
\]
we have $\cos\bigl((j - \alpha)h\bigr) \leq \cos(\alpha h) \leq 1$, and so
$\bigl|\cos\bigl((j - \alpha)h\bigr) - \cos(\alpha h)\bigr| \leq
\bigl|\cos\bigl((j - \alpha)h\bigr) - 1\bigr|$.  Thus, the denominator is
smallest in magnitude at $t = \pm \alpha h$.  Since the numerator of $f_1$ does
not vary with~$t$, we are done.

\emph{\textit{Case 2}:  $x_{-j}^* \leq x \leq -\alpha h$.}  Here, the
behavior of \eqref{EQN:OnePairUnmaxBound} is determined by $f_2$.  The
only critical point of the denominator $f_2$ in $[-\alpha h, \alpha h]$ is at
the left endpoint, where it takes the value $\cos(jh) - 1$.  At the right
endpoint, the denominator is $\cos(jh) - \cos(2\alpha h)$.  From
\[
0 < 2\alpha h < h \leq jh \leq Nh < \left(N + \frac{1}{2}\right)h = \pi,
\]
we see that $\cos(jh) \leq \cos(2\alpha h) \leq 1$, and so we have $|\cos(jh) -
\cos(\alpha h)| \leq |\cos(jh) - 1|$.  Thus, the denominator is smallest in
magnitude at $t = \alpha h$, and we are done, as in the previous case.

\emph{\textit{Case 3}:  $-\alpha h \leq x \leq 0$.}  In this case, for
$-\alpha h \leq \tilde{x}_0 \leq x$, the right-hand side of
\eqref{EQN:OnePairUnmaxBound} is governed by $f_3$, while for $x \leq
\tilde{x}_0 \leq \alpha h$, it is governed by $f_2$.  From the previous case,
we know that the maximum absolute value of $f_2(t)$ for $t \in [-\alpha h,
\alpha h]$ occurs at $t = \alpha h$, and a virtually identical argument shows
that the maximum absolute value of $f_3(t)$ over the same range occurs at $t =
-\alpha h$.  We are thus left to compare $|f_3(-\alpha h)|$ and $|f_2(\alpha
h)|$.  Since these two quantities have the same denominator, we need only
compare their numerators.  The conditions on $x$ imply that
\[
0 \leq \alpha h + x \leq \alpha h - x \leq 2\alpha h \leq jh < \pi,
\]
the later inequalities following as in the developments of the previous case.
Therefore, $\cos(jh) \leq \cos(x - \alpha h) \leq \cos(x + \alpha h)$, which
implies that $|\cos(jh) - \cos(x - \alpha h)| \leq |\cos(jh) - \cos(x + \alpha
h)|$. It follows that $|f_2(\alpha h)| \geq |f_3(-\alpha h)|$, as desired.
\end{proof}

\medskip

We can now prove the following result, which gives a bound on
$|\tilde{\ell}_0(x)|$ for $x \in [-\pi, 0]$ that is independent of the points
$\tilde{x}_j$.  First, we introduce some additional notation that we will need
for the remainder of our argument.  Define $x_0^* = 0$ and $x_{-N - 1}^* =
-\pi$.  For $0 \leq k \leq N$, let $R_k^* = [x_{-k - 1}^*, x_{-k}^*]$ and $R_k
= [(-k - 1 - \alpha)h, (-k + \alpha)h]$.  Observe that $\bigcup_{k = 0}^N R_k^*
= [-\pi, 0]$.  Again for $0 \leq k \leq N$, let
\[
P_k(x) = \prod_{j = 1}^N \left|\sin\left(\frac{x - (j - \alpha)h}{2}\right)\right|\ \times \prod_{j = 1}^k \left|\sin\left(\frac{x + (j - \alpha)h}{2}\right)\right| \times \prod_{j = k + 1}^N \left|\sin\left(\frac{x + (j + \alpha)h}{2}\right)\right|,
\]
and let
\[
Q_k = \prod_{j = 1}^N \left|\sin\left(\frac{(2\alpha - j)h}{2}\right)\right| \times \prod_{j = 1}^k \left|\sin\left(\frac{jh}{2}\right)\right| \times \prod_{j = k + 1}^N \left|\sin\left(\frac{(2\alpha + j)h}{2}\right)\right|.
\]
Define
\[
M_k = \max_{x \in [-\pi, 0] \cap R_k} \frac{P_k(x)}{Q_k},
\]
and note that $M_k$ does not depend on the points $\tilde{x}_j$.

\medskip

\begin{lemma}\label{LEM:l0CrazyProductBound}
For $0 \leq k \leq N$ and $x \in R_k^*$, we have $|\tilde{\ell}_0(x)| \leq
M_k$.
\end{lemma}

\medskip

\begin{proof}
Multiply together the inequalities derived in Lemma~\ref{LEM:MaxOverx0t}
for $1 \leq j \leq N$, and note that $R_k^* \subset R_k$ by
Lemma~\ref{LEM:xjsBounds}.
\end{proof}

\medskip

Next, we turn to bounding $M_k$.  Our strategy will be to reduce the products
$P_k(x)$ and $Q_k$ to sums by taking logarithms and then bounding the sums
using integrals.  We begin with $P_k(x)$, which requires more work than $Q_k$
because of its dependence on $x$.  The bound that we need is given by
Lemma~\ref{LEM:NumerBound}, but before presenting it, we first establish
several minor technical results that we will need in its proof.

\medskip
\begin{lemma}\label{LEM:bkckp1Bound}
For $0 \leq k \leq N$ and $x \in R_k$,
\[
\left|\sin\left(\frac{x + (k - \alpha)h}{2}\right)\sin\left(\frac{x + (k + 1 + \alpha)h}{2}\right)\right| \leq \left|\sin\left(\frac{(\alpha + 1/2)h}{2}\right)\right|^2.
\]
\end{lemma}

\begin{proof}
The derivative of the expression inside the absolute value signs on the
left-hand side of this inequality is $(1/2)\sin\bigl(x + (k + 1/2)h\bigr)$,
which vanishes inside $R_k$ only at $x = -(k + 1/2)h$.  The maximum absolute
value of the expression must occur at this point, since it is zero at the
endpoints of $R_k$.  Substituting this value in for $x$ in the left-hand side,
we arrive at the right-hand side.
\end{proof}

\medskip

\begin{lemma}\label{LEM:a1pb1Bound}
For $1 \leq k \leq N$ and $x \in R_k$,
\[
\left|\sin\left(\frac{x - (1 - \alpha)h}{2}\right)\sin\left(\frac{x + (1 - \alpha)h}{2}\right)\right| \geq \left|\sin\left(\frac{(k + 1 - 2\alpha)h}{2}\right)\sin\left(\frac{(k - 1)h}{2}\right)\right|.
\]
\end{lemma}

\begin{proof}
Let $f(x)$ be the expression inside the absolute value signs on the left-hand
side of this inequality.  Applying some trigonometric identities, we find that
$f(x) = \cos\bigl((1 - \alpha)h\bigr)/2 - \cos(x)/2$.  If $1 \leq k \leq N -
1$, then since
\[
0 \leq (1 - \alpha)h \leq (k - \alpha)h \leq -x \leq (k + 1 + \alpha)h \leq (N + \alpha)h < \pi,
\]
we have $\cos(x) \leq \cos\bigl((1 - \alpha)h\bigr)$, and so $f(x) \geq 0$ for
$x \in R_k$.  The same string of inequalities shows that $f'(x) = \sin(x)/2$ is
negative on $R_k$, so $f$ is decreasing on $R_k$.  Therefore, the smallest
absolute value of $f$ is obtained by evaluating at the right endpoint $x = (-k
+ \alpha)h$, and this produces the expression on the right-hand side of the
inequality to be established.

For the $k = N$ case, we note that $f$ has a critical point in $R_N$ at the
midpoint $x = -\pi$.  Since $f''(x) = \cos(x)/2$, we have $f''(-\pi) = -1/2$,
and so this point is a local maximum.  Thus, the minimum must occur at one of
the two endpoints.  Noting that $f$ is even about $\pi$, the value of $f$ must
be the same at both endpoints, so we may as well pick the right endpoint $x =
(-N + \alpha)h$.  Since $0 \leq (1 - \alpha)h \leq (N - \alpha)h \leq \pi$, the
value of $f$ at this endpoint is nonnegative, completing the proof.
\end{proof}

\medskip

\begin{lemma}\label{LEM:LooseTermBoundkeq0}
For $K \geq 3$ and $x \in R_0$,
\[
\left|\sin\left(\frac{x - (1 - \alpha)h}{2}\right)\sin\left(\frac{x + (1 + \alpha)h}{2}\right)\right| \leq \left|\sin\left(\frac{h}{2}\right)\right|^2.
\]
\end{lemma}

\begin{proof}
As in the previous argument, let $f(x)$ be the expression inside the absolute
value signs on the left-hand side of the inequality, and note that $f(x) =
\cos(h)/2 - \cos(x + \alpha h)/2$.  Since
\[
-\pi < -h \leq x + \alpha h \leq 2\alpha h \leq h < \pi,
\]
for $x \in R_0$, we have $\cos(h) \leq \cos(x + \alpha h)$ for $x \in R_0$, and
it follows that $f$ is negative on $R_0$.  Since $\cos(x + \alpha h) \leq 1$,
we have $0 \geq f(x) \geq \cos(h)/2 - 1/2$.  This lower bound is attained for
$x \in R_0$ at $x = -\alpha h$.  Thus, $f$ attains its maximum absolute value
on $R_0$ at $x = -\alpha h$, and substituting this value into the original
expression for $f$ yields the claimed inequality.
\end{proof}

\medskip

\begin{lemma}\label{LEM:LooseTermsBoundkeq1}
For $K \geq 3$ and $x \in R_1$, the following inequalities hold:
\begin{align*}
\left|\sin\left(\frac{x - (1 - \alpha)h}{2}\right)\right| &\geq \left|\sin\bigl((1 - \alpha)h\bigr)\right|, \\
\left|\sin\left(\frac{x + (1 - \alpha)h}{2}\right)\right| &\leq \left|\sin\left(\frac{(1 + 2\alpha)h}{2}\right)\right|, \\
\left|\sin\left(\frac{x + (2 + \alpha)h}{2}\right)\right| &\leq \left|\sin\left(\frac{(1 + 2\alpha)h}{2}\right)\right|.
\end{align*}
\end{lemma}

\begin{proof}
The first inequality follows from
\[
-\pi \leq -\frac{3}{2}h \leq \frac{x - (1 - \alpha)h}{2} \leq (\alpha - 1)h \leq 0,
\]
the second from
\[
-\pi \leq -\frac{(1 + 2\alpha)h}{2} \leq \frac{x + (1 - \alpha)h}{2} \leq 0,
\]
and the third from
\[
0 \leq \frac{x + (2 + \alpha)h}{2} \leq \frac{(1 + 2\alpha)h}{2} \leq \pi.
\]
\end{proof}

\medskip

\begin{lemma}\label{LEM:SingIntBound}
For $0 \leq k \leq N$ and $x \in R_k$,
\begin{multline*}
[x + (k - \alpha)h]\log\left(-\frac{x + (k - \alpha)h}{2}\right) - [x + (k + 1 + \alpha)h]\log\left(\frac{x + (k + 1 + \alpha)h}{2}\right) \\
\leq -(1 + 2\alpha)h\log\left(\frac{(\alpha + 1/2)h}{2}\right).
\end{multline*}
\end{lemma}

\begin{proof}
Let $f(x)$ be the expression on the left-hand side of this inequality.
The derivative of $f$ is
\[
f'(x) = \log\left(-\frac{x + (k - \alpha)h}{x + (k + 1 + \alpha)h}\right),
\]
and this vanishes in $R_k$ only at the point $x = -(k + 1/2)h$.  Since $f(x)$
tends to $-\infty$ as $x$ approaches the endpoints of $R_k$, $f$ must assume
its maximum value on $R_k$ at this point.  Evaluating $f$ at this point yields
the right-hand side of the claimed inequality.
\end{proof}

\medskip

\begin{lemma}\label{LEM:SingIntBoundkeq0}
For $x \in R_0$,
\begin{multline*}
\bigl(x - (1 - \alpha)h\bigr)\log\left(-\frac{x - (1 - \alpha)h}{2}\right) \\ - \bigl(x + (1 + \alpha)h\bigr)\log\left(\frac{x + (1 + \alpha)h}{2}\right) \leq -2h\log\left(\frac{h}{2}\right).
\end{multline*}
\end{lemma}

\begin{proof}
As in the previous argument, let $f(x)$ be the expression on the left-hand side
of the inequality.  We have
\[
f'(x) = \log\left(-\frac{x + (\alpha - 1)h}{x + (\alpha + 1)h}\right),
\]
and this vanishes in $R_0$ only at the point $x = -\alpha h$.  Moreover,
\[
f''(x) = \frac{2h}{(x + \alpha h)^2 - h^2}.
\]
The denominator of this function is a quadratic polynomial with positive
leading coefficient and zeroes at $(-1 - \alpha)h$ and $(1 - \alpha)h$.  Since
$x \in R_0$, we have $(-1 - \alpha)h \leq x \leq \alpha h < (1 - \alpha)h$, and
it follows that $f''$ is negative everywhere on $R_0$.  This implies that $f$
has a global maximum on $R_0$ at the critical point at $-\alpha h$ that we just
found.  Evaluating $f(-\alpha h)$ produces the right-hand side of the
inequality to be established.
\end{proof}

\medskip

\begin{lemma}\label{LEM:NumerBound}
For sufficiently large $K$ and $x \in R_k$, $0 \leq k \leq N$, we have
\[
P_k(x) \leq 5 \cdot 2^{-K} K
\]
for $k = 0, 1$ and
\[
P_k(x) \leq 3 \cdot 2^{-K} K^{2\alpha} \left|\sin\left(\frac{(k + 1 - 2\alpha)h}{2}\right)\sin\left(\frac{(k - 1)h}{2}\right)\right|^{\alpha - 1/2}
\]
for $2 \leq k \leq N$.
\end{lemma}

\begin{proof}
Let $S_k(x) = \log P_k(x)$.  For $1 \leq j \leq N$, define $a_j(x)$, $b_j(x)$,
and $c_j(x)$ by
\begin{align*}
a_j(x) &= \log\left|\sin\left(\frac{x - (j - \alpha)h}{2}\right)\right|, \\
b_j(x) &= \log\left|\sin\left(\frac{x + (j - \alpha)h}{2}\right)\right|, \\
c_j(x) &= \log\left|\sin\left(\frac{x + (j + \alpha)h}{2}\right)\right|.
\end{align*}
For brevity, we will typically suppress the argument when referring to these
quantities, writing $a_j$ in place of $a_j(x)$, and so forth.  Let
\begin{align*}
A_k(x) &= \sum_{j = 1}^{N - 1} \frac{1}{2} h(a_j + a_{j + 1}), \\
B_k(x) &= \sum_{j = 1}^{k - 1} \frac{1}{2}h(b_j + b_{j + 1}), \\
C_k(x) &= \sum_{j = k + 1}^{N - 1} \frac{1}{2}h(c_j + c_{j + 1}),
\end{align*}
and note that
\[
hS_k(x) = A_k(x) + B_k(x) + C_k(x) + \frac{1}{2}h(a_1 + a_N + b_1 + b_k + c_{k + 1} + c_N).
\]

The sums $A_k(x)$, $B_k(x)$, and $C_k(x)$ are composite trapezoidal rule
approximations to the integral of $\log\bigl|\sin\bigl((x + t)/2\bigr)\bigr|$
(with respect to $t$) over certain subintervals of $[-\pi, \pi]$.  Since this
function is concave-down everywhere on $[-\pi, \pi]$, these approximations will
yield lower bounds on the corresponding integrals \cite[p.\ 54]{DR1984_2e}.
More precisely, we have
\begin{align*}
A_k(x) &\leq \int_{-(N - \alpha)h}^{-(1 - \alpha)h} \log\left|\sin\left(\frac{x + t}{2}\right)\right| \: dt, \\
B_k(x) &\leq \int_{(1 - \alpha)h}^{(k - \alpha)h} \log\left|\sin\left(\frac{x + t}{2}\right)\right| \: dt, \\
C_k(x) &\leq \int_{(k + 1 + \alpha)h}^{(N + \alpha)h} \log\left|\sin\left(\frac{x + t}{2}\right)\right| \: dt,
\end{align*}
where the inequality for $B_k(x)$ holds for $1 \leq k \leq N$ and the
inequality for $C_k(x)$ holds for $0 \leq k \leq N - 1$.  We consider
four cases.

\emph{\textit{Case 1}:  $2 \leq k \leq N - 1$.}\ \hairspace In this case,
the preceding developments yield
\begin{multline*}
hS_k(x) \leq \int_{-\pi}^\pi - \int_{-\pi}^{-(N - \alpha)h} - \int_{-(1 - \alpha)h}^{(1 - \alpha)h} - \int_{(k - \alpha)h}^{(k + 1 + \alpha)h} - \int_{(N + \alpha)h}^\pi \log\left|\sin\left(\frac{x + t}{2}\right)\right| \: dt \\
+ \frac{1}{2}h(a_1 + a_N + b_1 + b_k + c_{k + 1} + c_N).
\end{multline*}
Now we just need to bound the integrals and loose terms on the right-hand side
of this inequality.  It turns out that the first integral can be evaluated
explicitly \cite[4.384-7]{GR2007_7e}:
\begin{equation}\label{EQN:IntLogSinto2}
\int_{-\pi}^\pi \log \left|\sin\left(\frac{x + t}{2}\right)\right| \: dt = -\pi\log(4).
\end{equation}
For the second and fifth integrals, we have the following bound, which can be
derived by applying the trapezoidal rule to the integral from $(N + \alpha)h$ to
$2\pi - (N - \alpha)h$ and using the periodicity of the integrand:
\begin{equation}\label{EQN:WraparoundTrapRule}
-\int_{-\pi}^{-(N - \alpha)h} - \int_{(N + \alpha)h}^\pi \log\left|\sin\left(\frac{x + t}{2}\right)\right| \: dt \leq -\frac{1}{2}h (a_{N} + c_{N}).
\end{equation}
The fourth integral requires some care, since it has a singularity in the
interval of integration at the point $t = -x$.  (Recall our assumption that $x
\in R_k = [(-k - 1 - \alpha)h, (-k + \alpha)h]$.)  We therefore split the
integral into two parts at that point.  Noting the expansion
\begin{equation}\label{EQN:logsinExpn}
\log\bigl(\sin(t)\bigr) = \log(t) - \frac{1}{6} t^2 - \frac{1}{180} t^4 + O(t^6)_{t \to 0^+},
\end{equation}
we have
\[
-\int_{(k - \alpha)h}^{-x} \log\left|\sin\left(\frac{x + t}{2}\right)\right| \: dt = \bigl(x + (k - \alpha)h\bigr)\left[\log\left(-\frac{x + (k - \alpha)h}{2}\right) - 1\right] + O(h^3)
\]
and
\begin{multline*}
-\int_{-x}^{(k + 1 + \alpha)h} \log\left|\sin\left(\frac{x + t}{2}\right)\right| \: dt \\
= \bigl(x + (k + 1 + \alpha)h\bigr)\left[1 - \log\left(\frac{x + (k + 1 + \alpha)h}{2}\right)\right] + O(h^3).
\end{multline*}
Adding these expressions together and applying
Lemma~\ref{LEM:SingIntBound}, we obtain
\[
-\int_{(k - \alpha)h}^{(k + 1 + \alpha)h} \log\left|\sin\left(\frac{x + t}{2}\right)\right| \: dt \leq (1 + 2\alpha)h -(1 + 2\alpha)h\log\left(\frac{(\alpha + 1/2)h}{2}\right) + O(h^3).
\]
For the third integral, we use another trapezoidal rule bound and combine the
result with the loose terms $(1/2)h(a_1 + b_1)$ to yield
\begin{multline}\label{EQN:BadIneqForkeq1}
-\int_{-(1 - \alpha)h}^{(1 - \alpha)h} \log \left|\sin\left(\frac{x + t}{2}\right)\right| \: dt + \frac{1}{2}h(a_1 + b_1)\leq \left(\alpha - \frac{1}{2}\right)h(a_1 + b_1) \\
\leq \left(\alpha - \frac{1}{2}\right)h\log\left|\sin\left(\frac{(k + 1 - 2\alpha)h}{2}\right)\sin\left(\frac{(k - 1)h}{2}\right)\right|,
\end{multline}
where the second inequality follows from Lemma~\ref{LEM:a1pb1Bound} and
the fact that $\alpha < 1/2$.  By Lemma~\ref{LEM:bkckp1Bound}
and \eqref{EQN:logsinExpn}, we now have
\begin{equation}\label{EQN:LooseTermsBound}
\frac{1}{2}h(b_k + c_{k + 1}) \leq h\log\left(\frac{(\alpha + 1/2)h}{2}\right) + O(h^3).
\end{equation}
Putting all of these results together, we conclude that
\begin{multline}\label{EQN:LogNumerBoundMainCase}
hS_k(x) \leq -\pi\log(4) + \left(\alpha - \frac{1}{2}\right)h\log\left|\sin\left(\frac{(k + 1 - 2\alpha)h}{2}\right)\sin\left(\frac{(k - 1)h}{2}\right)\right| \\
- 2\alpha h\log\left(\frac{(\alpha + 1/2)h}{2}\right) + (1 + 2\alpha)h + O(h^3).
\end{multline}
Dividing through by $h$, exponentiating, and suitably relaxing the constants
that emerge, we obtain the claimed bound in this case.

\emph{\textit{Case 2}:  $k = 1$.}\ \hairspace This case is similar to
the previous one.  In particular, all the same integral bounds apply except
that the second inequality in \eqref{EQN:BadIneqForkeq1} is meaningless
because the argument to the logarithm function vanishes.  We replace
\eqref{EQN:BadIneqForkeq1} and \eqref{EQN:LooseTermsBound} with
\begin{multline*}
-\int_{-(1 - \alpha)h}^{(1 - \alpha)h} \log \left|\sin\left(\frac{x + t}{2}\right)\right| \: dt + \frac{1}{2}h(a_1 + 2b_1 + c_2)\leq \left(\alpha - \frac{1}{2}\right)ha_1 + \frac{1}{2}hc_2 + \alpha h b_1 \\
\leq \left(\alpha - \frac{1}{2}\right)h\log\bigl((1 - \alpha)h\bigr) + \left(\alpha + \frac{1}{2}\right)h\log\left(\frac{(1 + 2\alpha)h}{2}\right) + O(h^3),
\end{multline*}
where the second inequality follows from
Lemma~\ref{LEM:LooseTermsBoundkeq1} and \eqref{EQN:logsinExpn}.
Combining this with the other results just established, we obtain
\begin{multline*}
hS_1(x) \leq -\pi\log(4) + \left(\alpha - \frac{1}{2}\right)h\log\bigl((1 - \alpha)h\bigr) + \left(\alpha + \frac{1}{2}\right)h\log\left(\frac{(1 + 2\alpha)h}{2}\right) \\
- (1 + 2\alpha) h\log\left(\frac{(\alpha + 1/2)h}{2}\right) + (1 + 2\alpha)h + O(h^3),
\end{multline*}
and this implies the claimed bound for this case.

\emph{\textit{Case 3}:  $k = N$.}\ \hairspace Since $C_N(x)$ has no terms,
we have, in this case,
\begin{multline*}
hS_N(x) \leq \int_{-\pi}^\pi - \int_{-\pi}^{-(N - \alpha)h} - \int_{-(1 - \alpha)h}^{(1 - \alpha)h} - \int_{(N - \alpha)h}^\pi \log\left|\sin\left(\frac{x + t}{2}\right)\right| \: dt \\
+ \frac{1}{2}h(a_1 + a_N + b_1 + b_N).
\end{multline*}
We can bound the third integral and the loose terms $(1/2)h(a_1 + b_1)$ using
\eqref{EQN:BadIneqForkeq1}; however, we cannot use
\eqref{EQN:WraparoundTrapRule} to bound the second and fourth integrals.
Instead, noting that there is a singularity at $-x$ (or a periodic image
thereof) within the domain of integration, we use \eqref{EQN:logsinExpn}
to find that
\[
-\int_{(N - \alpha)h}^{-x} \log\left|\sin\left(\frac{x + t}{2}\right)\right| \: dt = \bigl(x + (N - \alpha)h\bigr)\left[\log\left(-\frac{x + (N - \alpha)h}{2}\right) - 1\right] + O(h^3)
\]
and
\begin{multline*}
-\int_{-x}^{2\pi - (N - \alpha)h} \log\left|\sin\left(\frac{x + t}{2}\right)\right| \: dt \\
= \bigl(2\pi + x - (N - \alpha)h\bigr)\left[1 - \log\left(\frac{2\pi + x - (N - \alpha)h}{2}\right)\right] + O(h^3).
\end{multline*}
Noting that $2\pi + x - (N - \alpha)h = x + (N + 1 + \alpha)h$, we can add
these together and use periodicity and Lemma~\ref{LEM:SingIntBound} to
obtain
\begin{multline*}
-\int_{-\pi}^{-(N - \alpha)h} - \int_{(N - \alpha)h}^\pi \log\left|\sin\left(\frac{x + t}{2}\right)\right| \: dt \\
\leq (1 + 2\alpha)h - (1 + 2\alpha)h\log\left(\frac{(\alpha + 1/2)h}{2}\right) + O(h^3).
\end{multline*}
By the same identity, Lemma~\ref{LEM:bkckp1Bound}, and
\eqref{EQN:logsinExpn}, we have
\[
\frac{1}{2}h(a_N + b_N) \leq h \log\left(\frac{(\alpha + 1/2)h}{2}\right) + O(h^3).
\]
Putting everything together, we arrive once again at
\eqref{EQN:LogNumerBoundMainCase}, which finishes the argument in this
case.

\emph{\textit{Case 4}:  $k = 0$.}\ \hairspace As $B_0(x)$ has no terms, we
have
\begin{multline*}
hS_0(x) \leq \int_{-\pi}^\pi - \int_{-\pi}^{-(N - \alpha)h} - \int_{-(1 - \alpha)h}^{(1 + \alpha)h} - \int_{(N + \alpha)h}^\pi \log\left|\sin\left(\frac{x + t}{2}\right)\right| \: dt \\
+ \frac{1}{2}h(a_1 + a_N + c_1 + c_N).
\end{multline*}
We can take care of the second and fourth integrals and the loose terms
$(1/2)h(a_N + c_N)$ using \eqref{EQN:WraparoundTrapRule}.  For the third
integral, noting that $-x$ lies in the interval of integration, we use
\eqref{EQN:logsinExpn} one more time to conclude that
\[
-\int_{-(1 - \alpha)h}^{-x} \log\left|\sin\left(\frac{x + t}{2}\right)\right| \: dt = \bigl(x - (1 - \alpha)h\bigr)\left[\log\left(-\frac{x - (1 - \alpha)h}{2}\right) - 1\right] + O(h^3)
\]
and
\[
-\int_{-x}^{(1 + \alpha)h} \log\left|\sin\left(\frac{x + t}{2}\right)\right| \: dt = \bigl(x + (1 + \alpha)h\bigr)\left[1 - \log\left(\frac{x + (1 + \alpha)h}{2}\right)\right] + O(h^3).
\]
Adding these together and using Lemma~\ref{LEM:SingIntBoundkeq0}, we have
\[
-\int_{-(1 - \alpha)h}^{(1 + \alpha)h} \log\left|\sin\left(\frac{x + t}{2}\right)\right| \: dt \leq -2h\log\left(\frac{h}{2}\right) + 2h + O(h^3).
\]
By Lemma~\ref{LEM:LooseTermBoundkeq0} and \eqref{EQN:logsinExpn}, we have
\[
\frac{1}{2}h(a_1 + c_1) \leq h\log\left(\frac{h}{2}\right) + O(h^3).
\]
Assembling all these facts, we find that
\[
hS_0(x) \leq -\pi \log(4) - h\log\left(\frac{h}{2}\right) + 2h + O(h^3),
\]
and upon dividing through by $h$, exponentiating, and adjusting the constant
factors that arise, we obtain the desired result.

All cases have now been handled.  The proof is complete.
\end{proof}

\medskip

Next, we bound $Q_k$.  The result we need is the following:

\medskip

\begin{lemma}\label{LEM:DenomBound}
For sufficiently large $K$,
\[
Q_k \geq (1 - 2\alpha) 2^{-K} K^{1 - 2\alpha} \left|\sin\left(\frac{(k + 1/2 + \alpha)h}{2}\right)\right|^{-2\alpha}.
\]
\end{lemma}

\begin{proof}
The proof is similar in structure to that of Lemma~\ref{LEM:NumerBound}.
Let $S_k = \log(Q_k)$, so that
\begin{equation}\label{EQN:QkBoundSums}
S_k = \sum_{j = 1}^N \log\left|\sin\left(\frac{(2\alpha - j)h}{2}\right)\right| + \sum_{j = 1}^k \log\left|\sin\left(\frac{jh}{2}\right)\right| + \sum_{j = k + 1}^N \log\left|\sin\left(\frac{(2\alpha + j)h}{2}\right)\right|.
\end{equation}
We will bound $S_k$ using integrals of $\log\bigl|\sin(t/2)\bigr|$, just as
before; but this time, since we seek a lower bound, we use the midpoint rule
instead of the trapezoidal rule \cite[p.\ 54]{DR1984_2e}.  Assuming $0
\leq \alpha \leq 1/4$, we have
\begin{equation}\label{EQN:QkBoundInts}
hS_k \geq \int_{-\pi}^\pi - \int_{(2\alpha - 1/2)h}^{h/2} - \int_{(k + 1/2)h}^{(k + 2\alpha + 1/2)h} \log\left|\sin\left(\frac{t}{2}\right)\right| \: dt.
\end{equation}
We evaluated the first integral in \eqref{EQN:IntLogSinto2}, above.  We
bound the third integral using the midpoint rule:
\[
-\int_{(k + 1/2)h}^{(k + 2\alpha + 1/2)h} \log\left|\sin\left(\frac{t}{2}\right)\right| \: dt \geq -2\alpha h \log\left|\sin\left(\frac{(k + 1/2 + \alpha)h}{2}\right)\right|.
\]
For the second integral, we split the interval of integration at the
singularity at $0$ and use \eqref{EQN:logsinExpn} to compute
\begin{multline*}
-\int_{(2\alpha - 1/2)h}^{h/2} \log\left|\sin\left(\frac{t}{2}\right)\right| \: dt = \left(2\alpha - \frac{1}{2}\right)h\log\left(\frac{(1/2 - 2\alpha)h}{2}\right) \\ - \frac{h}{2} \log\left(\frac{h}{4}\right) + (1 - 2\alpha)h + O(h^3).
\end{multline*}
From these results, it follows that
\begin{multline*}
hS_k \geq -\pi \log(4) -2\alpha h \log\left|\sin\left(\frac{(k + 1/2 + \alpha)h}{2}\right)\right| \\
+ \left(2\alpha - \frac{1}{2}\right)h\log\left(\frac{(1/2 - 2\alpha)h}{2}\right) - \frac{h}{2} \log\left(\frac{h}{4}\right) + (1 - 2\alpha)h + O(h^3).
\end{multline*}
Dividing through by $h$, exponentiating, and suitably adjusting the constant
factors that arise, we obtain the claimed result.

If $1/4 < \alpha < 1/2$, the argument is similar except that we have to track
the $j = 1$ term in the first sum in the definition of $S_k$ independently.  We
write
\[
hS_k \geq \int_{-\pi}^{\pi} - \int_{(2\alpha - 3/2)h}^{h/2} - \int_{(k + 1/2)h}^{(k + 2\alpha + 1/2)h} \log\left|\sin\left(\frac{t}{2}\right)\right| \: dt + h\log\left|\sin\left(\frac{(2\alpha - 1)h}{2}\right)\right|.
\]
Using \eqref{EQN:logsinExpn} one last time, we compute
\[
h\log\left|\sin\left(\frac{(2\alpha - 1)h}{2}\right)\right| = h \log\left(\frac{(1 - 2\alpha)h}{2}\right) + O(h^3).
\]
and
\begin{multline*}
-\int_{(2\alpha - 3/2)h}^{h/2} \log\left|\sin\left(\frac{t}{2}\right)\right| \: dt = \left(2\alpha - \frac{3}{2}\right)\log\left(\frac{(3/2 - 2\alpha)h}{2}\right) \\ - \frac{h}{2}\log\left(\frac{h}{4}\right) + 2(1 - \alpha)h + O(h^3).
\end{multline*}
Therefore,
\begin{multline*}
hS_k \geq -\pi\log(4) -2\alpha h \log\left|\sin\left(\frac{(k + 1/2 + \alpha)h}{2}\right)\right| + h \log\left(\frac{(1 - 2\alpha)h}{2}\right) \\
+ \left(2\alpha - \frac{3}{2}\right)\log\left(\frac{(3/2 - 2\alpha)h}{2}\right) - \frac{h}{2}\log\left(\frac{h}{4}\right) + 2(1 - \alpha)h + O(h^3).
\end{multline*}
and this implies the claimed bound in the usual way.
\end{proof}

\medskip

Note that the proof of this lemma shows that the constant $1 - 2\alpha$ can be
dropped from the bound when $0 \leq \alpha \leq 1/4$.  We have chosen for
simplicity to include it in this case anyway because omitting it will at best
improve our final results by a small constant factor.

At this point, we have all that we need to bound $|\tilde{\ell}_0(x)|$
uniformly for $x \in [-\pi, \pi]$ and independent of the points $\tilde{x}_j$:
evaluate the bounds on $M_k$ for $x \in [-\pi, 0] \cap R_k$ given by
Lemmas~\ref{LEM:NumerBound} and \ref{LEM:DenomBound}, and take the
maximum over $k$.  Lemma~\ref{LEM:l0CrazyProductBound} shows that the
result bounds $|\tilde{\ell}_0(x)|$ for $x \in [-\pi, 0]$.  By symmetry, the
same bound must hold for $x \in [0, \pi]$ as well.  Even further, by
considering circular rotations of the points $\tilde{x}_j$, the bound can be
seen to apply to $|\tilde{\ell}_k(x)|$ for $k \neq 0$.  Therefore, by
\eqref{EQN:PertLebConst}, we could bound $\tilde{\Lambda}_N$ by
multiplying the bound on $|\tilde{\ell}_0(x)|$ by $K$.

We can do better than this, however, because
Lemmas~\ref{LEM:l0CrazyProductBound} and \ref{LEM:NumerBound} retain some
information about how $|\tilde{\ell}_0(x)|$ varies with $x$ through the
hypothesis that $x \in R_k$.  We can use this information to get a better bound
on $|\tilde{\ell}_k(x)|$ for $k \neq 0$ than the one just described.  The
result we need is given by the following lemma, which we could have proved
earlier but have delayed until now.

\medskip

\begin{lemma}\label{LEM:lkBounds}
If $x \in R_p^*$, $0 \leq p \leq N$, then for $-N \leq k \leq N$,
\begin{equation}\label{EQN:lkBounds}
|\tilde{\ell}_k(x)| \leq \begin{cases}
\max(M_{-(p + k)}, M_{-(p + k + 1)}, M_{-(p + k + 2)}) & -N \leq p + k \leq -2 \\
\max(M_0, M_1) & p + k = -1, 0 \\
\max(M_{p + k - 1}, M_{p + k}, M_{p + k + 1}) & 1 \leq p + k \leq N - 1 \\
\max(M_{N - 1}, M_{N}) & p + k = N \\
\max(M_{K - (p + k)}, M_{K - (p + k + 1)}, M_{K - (p + k + 2)}) & N + 1 \leq p + k \leq 2N - 1 \\
\max(M_0, M_1) & p + k = 2N.
\end{cases}
\end{equation}
\end{lemma}

\begin{proof}
For $k = 0$, the result follows from
Lemma~\ref{LEM:l0CrazyProductBound}, which actually gives a stronger
bound.  The proof for $k \neq 0$ is ultimately just a matter of reducing it to
the $k = 0$ case by exploiting circular and reflectional symmetry; however,
there are some subtleties, so we will spell out the details to make things
clear.  Note that $x \in R_p^*$ implies $x \in R_p$ by
Lemma~\ref{LEM:xjsBounds}.

First, suppose that $1 \leq k \leq N$.  Then, $1 \leq p + k \leq 2N$, so only
the last four cases in \eqref{EQN:lkBounds} are relevant.  Let
\[
\hat{x}_j = \begin{cases}
\tilde{x}_{j + k} - kh            & -N \leq j \leq N - k \\
\tilde{x}_{j + k - K} + 2\pi - kh & N - k + 1 \leq j \leq N.
\end{cases}
\]
These points are just a circular shift in $[-\pi, \pi]$ of the points
$\tilde{x}_j$ by $kh$.  It follows that $\tilde{\ell}_k(x) = \hat{\ell}_0(x -
kh)$, where $\hat{\ell}_0$ is the (trigonometric) Lagrange basis function for
the points $\hat{x}_j$ that takes on the value $1$ at $\hat{x}_0$.  One can
easily check that
\[
\hat{x}_j = \begin{cases}
x_j + t_{j + k}h     & -N \leq j \leq N - k \\
x_j + t_{j + k - K}h & N - k + 1 \leq j \leq N,
\end{cases}
\]
where the $x_j$ are the equispaced points \eqref{EQN:EquiPts}, and the
$t_j$ are defined by \eqref{EQN:PertPts}.  Thus, the points $\hat{x}_j$
constitute a set of perturbed equispaced points of the sort that we have been
considering.  In particular, we can use
Lemma~\ref{LEM:l0CrazyProductBound} to bound $\hat{\ell}_0(x - kh)$ and
hence $\tilde{\ell}_k(x)$.  We consider several cases.

\emph{\textit{Case 1}:  $1 \leq p + k \leq N - 1$.} Since $x \in R_p$, it
follows that $x - kh \in R_{p + k}$, which means that $x - kh$ must belong to
one of $R_{p + k - 1}^*$, $R_{p + k}^*$, and $R_{p + k + 1}^*$, again by
Lemma~\ref{LEM:xjsBounds}.  By
Lemma~\ref{LEM:l0CrazyProductBound}, $|\hat{\ell}_0(x - kh)| \leq
\max(M_{p + k - 1}, M_{p + k}, M_{p + k + 1})$.

\emph{\textit{Case 2}:  $p + k = N$ and $(-p - 1/2)h \leq x \leq (-p +
\alpha)h$.}  We have $x - kh \in R_N$.  Moreover, $x - kh \geq (-p - k - 1/2)h
= (-N - 1/2)h = -\pi$, so $x - kh \in [-\pi, 0] \cap R_N$.  Thus, $x - kh$
belongs to either $R_N^*$ or $R_{N - 1}^*$ by Lemma~\ref{LEM:xjsBounds},
and so Lemma~\ref{LEM:l0CrazyProductBound} gives $|\hat{\ell}_0(x -
kh)| \leq \max(M_{N - 1}, M_N)$.

\emph{\textit{Case 3}:  $p + k = N$ and $(-p - 1 - \alpha)h \leq x < (-p -
1/2)h$.}  Again, we have $x - kh \in R_N$, but this time, $x - kh < \pi$.
Nevertheless, $\hat{\ell}_0(x - kh) = \hat{\ell}_0(x - kh + 2\pi)$, and $x - kh
+ 2\pi \in [0, \pi] \cap -R_N$.  By reflecting the problem about $0$ (i.e.,
replacing $\hat{x}_j$ with $-\hat{x}_j$ for each $j$ and $x - kh + 2\pi$ by
$-(x - kh + 2\pi) \in [-\pi, 0] \cap R_N$), and applying
Lemma~\ref{LEM:l0CrazyProductBound}, we obtain $|\hat{\ell}_0(x - kh)|
\leq \max(M_{N - 1}, M_N)$ as in the previous case.

\emph{\textit{Case 4}:  $N + 1 \leq p + k \leq 2N - 1$.}  Just as in the
previous case, we will look not at $\hat{\ell}_0(x - kh)$ but at
$\hat{\ell}_0(x - kh + 2\pi)$.  Noting that $2\pi = Kh$, we see that $x - kh +
2\pi \in -R_{K - (p + k + 1)}$.  Since $x \geq -\pi$ and $k \leq N$, we have $x
- kh + 2\pi \geq -\pi + (K - N)h = h/2 > 0$.  Thus, $x - kh + 2\pi \in [0, \pi]
\cap -R_{K - (p + k + 1)}$.  Reflecting about $0$ as was done in the previous
case and noting that $-(x - kh + 2\pi)$ must belong to one of $R_{K - (p +
k)}^*$, $R_{K - (p + k + 1)}^*$, and $R_{K - (p + k + 2)}^*$ by
Lemma~\ref{LEM:xjsBounds}, we may apply
Lemma~\ref{LEM:l0CrazyProductBound} to conclude that $|\hat{\ell}_0(x -
kh)| \leq \max(M_{K - (p + k)}, M_{K - (p + k + 1)}, M_{K - (p + k + 2)})$.

\emph{\textit{Case 5}:  $p + k = 2N$.}  This is handled exactly the same
as the previous case except that since $x - kh + 2\pi \in [0, \pi] \cap -R_0$,
we have that $-(x - kh + 2\pi)$ can belong only to one of $R_0^*$ and $R_1^*$.
Therefore, $|\hat{\ell}_0(x - kh)| \leq \max(M_0, M_1)$.

For $-N \leq k \leq -1$, the argument is similar.  In this case, the circularly
shifted points $\hat{x}_j$ are
\[
\hat{x}_j = \begin{cases}
\tilde{x}_{j + k} - kh            & -N - k \leq j \leq N \\
\tilde{x}_{j + k + K} - 2\pi - kh & -N \leq j \leq -N - k - 1,
\end{cases}
\]
so that
\[
\hat{x}_j = \begin{cases}
x_j + t_{j + k}h     & -N - k \leq j \leq N \\
x_j + t_{j + k + K}h & -N \leq j \leq -N - k - 1.
\end{cases}
\]
Just as before, we have $\tilde{\ell}_k(x) = \hat{\ell}_0(x - kh)$.  Noting
that $-N \leq p + k \leq N - 1$, the proof again breaks into cases as follows.

\emph{\textit{Case 1}:  $1 \leq p + k \leq N - 1$.}  Just as in the
previous Case 1, we have $x - kh \in R_{p + k}$, and the result follows in
exactly the same way.

\emph{\textit{Case 2}:  $p + k = 0$ and $(-p - 1 -\alpha)h \leq x \leq
-ph$.}  Here, $x - kh \in R_0$, and the restriction on $x$ forces $x - kh \leq
0$, so in fact, $x - kh \in [-\pi, 0] \cap R_0$.  Therefore, $x - kh$ belongs
to one of $R_0^*$ and $R_1^*$ by Lemma~\ref{LEM:xjsBounds}, and so by
Lemma~\ref{LEM:l0CrazyProductBound} we have $|\hat{\ell}_0(x - kh)|
\leq \max(M_0, M_1)$.

\emph{\textit{Case 3}: $p + k = 0$ and $-ph < x \leq (-p + \alpha)h$.} Now
$x - kh \in R_0$, but $0 < x - kh \leq \alpha h$.  To bound $\hat{\ell}_0(x -
kh)$ in this case, we reflect the problem about $0$ as we did in some of the
cases for positive $k$ above.  Since $[-\alpha h, \alpha h] \subset R_0$, we
have $-(x - kh) \in [-\pi, 0] \cap R_0$, and so
Lemma~\ref{LEM:l0CrazyProductBound} tells us that $|\hat{\ell}_0(x -
kh)| \leq \max(M_0, M_1)$ once again.

\emph{\textit{Case 4}: $p + k = -1$ and $(-p - 1 - \alpha)h \leq x \leq
(-p - 1)h$.} In this case, $x - kh \in [-\alpha h, 0]$ and hence belongs to
$[-\pi, 0] \cap R_0$.  Applying Lemma~\ref{LEM:l0CrazyProductBound}, we
have $|\hat{\ell}_0(x - kh)| \leq \max(M_0, M_1)$ just as in the previous two
cases.

\emph{\textit{Case 5}: $p + k = -1$ and $(-p - 1)h < x \leq (-p +
\alpha)h$.} Now, $x - kh \in [0, \pi] \cap -R_0$.  Reflecting in $0$ and using
Lemma~\ref{LEM:l0CrazyProductBound} yet again, we have $|\hat{\ell}_0(x -
kh)| \leq \max(M_0, M_1)$.

\emph{\textit{Case 6}: $-N \leq p + k \leq -2$.}  We have $x - kh \in
-R_{-(p + k + 1)}$.  Since $-R_{-(p + k + 1)} \subset [0, \pi]$, we reflect in
$0$ and observe that, by Lemma~\ref{LEM:xjsBounds}, $-(x - kh)$ belongs
to one of $R_{-(p + k)}^*$, $R_{-(p + k + 1)}^*$, and $R_{-(p + k + 2)}^*$.
Applying Lemma~\ref{LEM:l0CrazyProductBound} one last time, we obtain
$|\hat{\ell}_0(x - kh)| \leq \max(M_{-(p + k)}, M_{-(p + k + 1)}, M_{-(p + k +
2)})$.

All cases have been handled.  The proof is finished.
\end{proof}

\medskip

The point of Lemma~\ref{LEM:lkBounds} is that it allows us to bound
$\tilde{\Lambda}_N$ by summing the bounds of
Lemma~\ref{LEM:l0CrazyProductBound} over $k$ instead of maximizing them
over $k$ and multiplying by $K$ as described previously.

\medskip

\begin{lemma}
We have
\begin{equation}\label{EQN:LebConstSumBound}
\tilde{\Lambda}_N \leq 9\sum_{k = 0}^N M_k.
\end{equation}
\end{lemma}

\begin{proof}
Suppose that $x \in [-\pi, 0] \cap R_p^*$, $0 \leq p \leq N$.  We can use
Lemma~\ref{LEM:lkBounds} to bound the sum in
\eqref{EQN:PertLebConst} for this value of $x$ by summing the right-hand
side of \eqref{EQN:lkBounds} over $-N \leq k \leq N$.  This is
equivalent to summing it over the values of $p + k$ such that $-N + p \leq p +
k \leq N + p$, and this is certainly bounded above by the sum over the larger
range $-N \leq p + k \leq 2N$.  Writing $j$ in place of $p + k$, it follows
that
\begin{multline*}
\sum_{k = -N}^N |\tilde{\ell}_k(x)| \leq \sum_{j = -N}^{-2} \max(M_{-j}, M_{-(j + 1)}, M_{-(j + 2)}) + \sum_{j = 1}^{N - 1} \max(M_{j - 1}, M_j, M_{j + 1}) \\
+ \sum_{j = N + 1}^{2N - 1} \max(M_{K - j}, M_{K - (j + 1)}, M_{K - (j + 2)}) \\
+ 3\max(M_0, M_1) + \max(M_{N - 1}, M_N).
\end{multline*}
Since $\max(a, b) \leq a + b$ when $a, b \geq 0$, we can convert the maxima
into sums to obtain
\begin{multline*}
\sum_{k = -N}^N |\tilde{\ell}_k(x)| \leq \sum_{j = 2}^N M_j + \sum_{j = 1}^{N - 1} M_j + \sum_{j = 0}^{N - 2} M_j + \sum_{j = 0}^{N - 2} M_j + \sum_{j = 1}^{N - 1} M_j + \sum_{j = 2}^N M_j \\
+ \sum_{j = 2}^N M_j + \sum_{j = 1}^{N - 1} M_j + \sum_{j = 0}^{N - 2} M_j + 3M_0 + 3M_1 + M_{N - 1} + M_N
\end{multline*}
after simplifying the indices of summation.  We immediately obtain
\[
\sum_{k = -N}^N |\tilde{\ell}_k(x)| \leq 9\sum_{j = 0}^N M_j.
\]
Since the right-hand side of this inequality is independent of $p$, this bound
actually holds for all $x \in [-\pi, 0]$.  Even further, since the $M_j$ are
independent of both $x$ and the points \eqref{EQN:PertPts}, by symmetry,
it holds for all $x \in [-\pi, \pi]$.  The result now follows from
\eqref{EQN:PertLebConst}.
\end{proof}

\medskip

We can now prove Theorem~2.1 from the main article.

\medskip

\emph{Proof of Theorem~2.1.}
We use Lemmas~\ref{LEM:NumerBound} and \ref{LEM:DenomBound} to
bound the right-hand side of \eqref{EQN:LebConstSumBound}.  For $K$
sufficiently large and $k = 0, 1$, we have
\begin{multline}\label{EQN:MkBoundk01}
M_k \leq \frac{5}{1 - 2\alpha} K^{2\alpha}\left|\sin\left(\frac{(k + 1/2 + \alpha)\pi}{K}\right)\right|^{2\alpha} \\
\leq \frac{5}{1 - 2\alpha} K^{2\alpha}\left|\frac{(k + 1/2 + \alpha)\pi}{K}\right|^{2\alpha} \leq \frac{10\pi}{1 - 2\alpha},
\end{multline}
while for $2 \leq k \leq N$,
\begin{multline*}
M_k \leq \frac{3}{1 - 2\alpha} K^{4\alpha - 1} \frac{\left|\sin\left(\frac{(k + 1/2 + \alpha)\pi}{K}\right)\right|^{2\alpha}}{\left|\sin\left(\frac{(k + 1 - 2\alpha)\pi}{K}\right)\sin\left(\frac{(k - 1)\pi}{K}\right)\right|^{1/2 - \alpha}}
\\ \leq \frac{3}{1 - 2\alpha} K^{4\alpha - 1} \frac{\left|\sin\left(\frac{(k + 1/2 + \alpha)\pi}{K}\right)\right|^{2\alpha}}{\left|\sin\left(\frac{(k - 1)\pi}{K}\right)\right|^{1 - 2\alpha}}.
\end{multline*}
In deriving the last expression, we have used the inequality
\[
\left|\sin\left(\frac{(k + 1 - 2\alpha)\pi}{K}\right)\right| \geq \left|\sin\left(\frac{(k - 1)\pi}{K}\right)\right|,
\]
which clearly holds for $2 \leq k \leq N - 1$ and for $k = N$ with $1/4 \leq
\alpha < 1/2$, since in those cases, $(k + 1 - 2\alpha)\pi/K \in [0, \pi/2]$,
and $k + 1 - 2\alpha \geq k > k - 1$.  To see that it holds for $k = N$ with $0
< \alpha < 1/4$ as well, note that in this case
\[
\sin\left(\frac{(N + 1 - 2\alpha)\pi}{K}\right) = \sin\left(\frac{(N + 2\alpha)\pi}{K}\right)
\]
by the symmetry of sine about $\pi/2$.  Since $N + 2\alpha \in [0, \pi/2]$ and
$N + 2\alpha > N - 1$, the inequality follows.

Using the inequalities $|\sin(x)| \leq |x|$ for $x \in \R$ and $|\sin(x)| \geq
(2/\pi)|x|$ for $|x| \leq \pi/2$, we can simplify the bound on $M_k$ for $2
\leq k \leq N$ even further to
\begin{equation}\label{EQN:MkBoundkge2}
M_k \leq \frac{3}{1 - 2\alpha} K^{4\alpha - 1} \frac{\left|\frac{(k + 1/2 + \alpha)\pi}{K}\right|^{2\alpha}}{\left|\frac{2(k - 1)}{K}\right|^{1 - 2\alpha}} \leq \frac{3\pi}{1 - 2\alpha} \frac{(k + 1)^{2\alpha}}{(k - 1)^{1 - 2\alpha}}.
\end{equation}
The result now follows from summing the bounds on the $M_k$ established in
\eqref{EQN:MkBoundk01} and \eqref{EQN:MkBoundkge2} and bounding
the sum by interpreting it as a midpoint rule approximation%
\footnote{We thank Andrew Thompson for suggesting the use of the
midpoint rule instead of a simpler Riemann sum.  The latter yields a bound that
does not have $O(\log K)$ behavior in the limit as $\alpha \to 0$.}
to the integral of a function that is concave-up (note that $N + 1/2 = K/2$):
\begin{multline*}
\sum_{k = 2}^N \frac{(k + 1)^{2\alpha}}{(k - 1)^{1 - 2\alpha}} \leq \int_{3/2}^{N + 1/2} \frac{(x + 1)^{2\alpha}}{(x - 1)^{1 - 2\alpha}} \: dx \\
\leq (K/2 + 1)^{2\alpha} \int_{3/2}^{K/2} \frac{dx}{(x - 1)^{1 - 2\alpha}} = \frac{(K^2/4 - 1)^{2\alpha} - (K/4 + 1/2)^{2\alpha}}{2\alpha} \\
\leq \frac{(K^2/4)^{2\alpha} - (K/4)^{2\alpha}}{2\alpha} = \frac{K^{4\alpha} - K^{2\alpha}}{4^{2\alpha} 2\alpha} \leq \frac{K^{4\alpha} - 1}{2\alpha}.
\end{multline*}
\endproof

\medskip

We close with a word about why our argument falls short of establishing the
stronger bound on $\tilde{\Lambda}_N$ that we conjecture involving
$N^{2\alpha}$ instead of $N^{4\alpha}$.  As summarized in the opening
paragraphs of this appendix, our argument proceeds by choosing the perturbed
points $\tilde{x}_j$ to maximize $|\tilde{\ell}_k|$ for a fixed value of $k$,
bounding the maximum, and then summing the bounds.  This is a different (and
easier) problem than choosing the points to maximize the sum $\sum_{k = -N}^N
|\tilde{\ell}_k|$ and bounding that maximum instead.

In symbols, our argument bounds $\tilde{\Lambda}_N$ by bounding the rightmost
expression in the following chain of inequalities:
\[
\tilde{\Lambda}_N \leq \max_{\tilde{x}_{-N}, \ldots, \tilde{x}_N} \max_{x \in [-\pi, \pi]} \sum_{k = -N}^N |\tilde{\ell}_k(x)| \leq \max_{x \in [-\pi, \pi]} \sum_{k = -N}^N \max_{\tilde{x}_{-N}, \ldots, \tilde{x}_N} |\tilde{\ell}_k(x)|.
\]
The loss enters in the passage to the rightmost expression from the one in the
middle.  To prove the stronger bound, one needs to consider the
$|\tilde{\ell}_k|$ all together at once in the sum instead of individually as
we have done here.

\bibliography{bib/jshort,bib/master,local.bib}
\bibliographystyle{siam}